%% file: randomized_krylov_for_odes.tex
\documentclass[final,dvipsnames,leqno, letterpaper]{etna}

\input{setup.tex}

\title{A general framework for Krylov ODE residuals with applications to randomized Krylov methods\thanks{This work was supported by Deutsche Forschungsgemeinschaft through DFG grant 538686094 -- ``Matrix functions via randomized sketching''.}}
\author{Emil Krieger\thanks{School of Mathematics and Natural Sciences, Bergische Universit\"at Wuppertal, 42097 Wuppertal, Germany, \texttt{\{krieger,marcel\}@uni-wuppertal.de}.} \and Marcel Schweitzer$^\dagger$\!\!\!}

\begin{document}
\maketitle

\pagestyle{myheadings} \thispagestyle{plain}
\markboth{E.~KRIEGER AND M.~SCHWEITZER}{RESIDUAL-BASED RANDOMIZED KRYLOV METHODS FOR ODES}

\begin{abstract}
Randomized Krylov subspace methods that employ the sketch-and-solve paradigm to substantially reduce orthogonalization cost have recently shown great promise in speeding up computations for many core linear algebra tasks (e.g., solving linear systems, eigenvalue problems and matrix equations, as well as approximating the action of matrix functions on vectors) whenever a nonsymmetric matrix is involved. An important application that requires approximating the action of matrix functions on vectors is the implementation of exponential integration schemes for ordinary differential equations. In this paper, we specifically analyze randomized Krylov methods from this point of view. In particular, we use the residual of the underlying differential equation to derive a new, reliable a posteriori error estimate that can be used to monitor convergence and decide when to stop the iteration. To do so, we first develop a very general framework for Krylov ODE residuals that unifies existing results, simplifies their derivation and allows extending the concept to a wide variety of methods beyond randomized Arnoldi (e.g., rational Krylov methods, Krylov methods using a non-standard inner product, \dots). In addition, we discuss certain aspects regarding the efficient implementation of sketched Krylov methods. Numerical experiments on large-scale ODE models from real-world applications illustrate the use of the sketched residual norm as stopping criterion as well as the general competitiveness of sketched Krylov methods for ODEs in comparison to other Krylov-based methods.
\end{abstract}



\section{Introduction}\label{sec:intro}
Exponential time integration schemes~\cite{HochbruckOstermann2010} are among the most promising class of numerical methods for the solution of stiff, large-scale systems of ordinary differential equations (ODEs). Their implementation crucially relies on the ability to efficiently evaluate the action of matrix functions~\cite{Higham2008} on vectors---the action of the matrix exponential and $\varphi$-functions in the case of first-order problems~\cite{HochbruckOstermann2010} and compositions of square roots, inverses, and trigonometric functions in the case of second-order problems~\cite{Gautschi1961,HochbruckLubich1999}. The growing success of and interest in exponential integrators can therefore in part also be attributed to algorithmic advances for matrix functions in the last two decades~\cite{al2011computing,fasi2024challenges,FrommerGuettelSchweitzer2014a,GuettelKressnerLund2020,HighamAlMohy2010}.

When $A \in \Rnn$ is large and sparse (or otherwise structured), the matrix function $f(A)$ cannot be computed explicitly due to enormous computational cost and memory requirements. Instead, one aims to directly approximate $f(A)\vb$, where $\vb \in \Rn$ is a given vector, by an iterative method---typically a Krylov subspace method. When $A$ is of moderate size or highly structured, such that shifted linear systems $(A+\sigma I)\vx = \vb$ can be efficiently treated by a direct solver, rational Krylov subspace methods are often the method of choice~\cite{Guettel2010,GuettelKnizhnerman2013}. In other cases---also, e.g., when $A$ is given only implicitly by some routine that returns the result of the matrix-vector product $\vv \mapsto A\vv$---one is restricted to the use of polynomial Krylov methods~\cite{DruskinKnizhnerman1989,GaudreaultRainwaterTokman2018,Saad1992,sidj98}, which typically converge much more slowly. This can be a huge problem---in particular when $A$ is nonsymmetric---as storage requirements and orthogonalization cost grow with the number of iterations. 

One possible remedy for this problem is the use of an \emph{incomplete orthogonalization} within Krylov methods (instead of the usual, full Gram--Schmidt orthogonalization that is typically used in the Arnoldi method), which has been demonstrated to work quite successfully in the context of exponential integrators~\cite{GaudreaultRainwaterTokman2018,Koskela2015}. It does, however, typically lead to a further delay in convergence and can necessitate the use of restarting (or ``sub time stepping'') techniques in order to reach the desired accuracy.

A rather recent new avenue in algorithms for approximating the action of a matrix function on a vector is the use of \emph{randomized methods}, in particular those based on the \emph{sketch-and-solve paradigm}~\cite{MartinssonTropp2020,woodruff2014sketching}. In the context of Krylov methods, these approaches typically aim at substantially reducing orthogonalization cost, while avoiding the delay of convergence observed in methods with incomplete orthogonalization; see, e.g.,~\cite{balabanov2022randomized,burke2025gmres,cortinovis2024speeding,guidotti2025accelerating,GuettelSchweitzer2023,guttel2024sketch,nakatsukasa2024fast,palitta2025sketched,palitta2025mateq}, and the references therein. 

While these methods experimentally show very promising results, they are still in their infancy. There is still a large gap in comparison to classical Krylov methods without randomization, both regarding their theoretical understanding and their adoption in production-level codes. One particular shortcoming that might prevent a more widespread adoption is the lack of reliable error estimates and stopping criteria. Most publications on the topic so far either completely omit a discussion of this topic or employ a heuristic error estimate based on the difference of iterates~\cite{guidotti2025accelerating,GuettelSchweitzer2023}. Such an estimate can already be unreliable for classical Krylov methods, and even more so in the presence of randomization.

For matrix functions that are connected to an underlying (first- or second-order) ODE, an established way to monitor the progress of the method is to keep track of the \emph{ODE residual}. For classical Krylov methods, it is well known that the residual (or its norm) can be cheaply computed in terms of quantities arising naturally from the Arnoldi process~\cite{bgh13,bokn20,bks23,bkt21,celledoni1997krylov}. In this paper, we introduce a general framework for Krylov residuals, which in particular allows us to adapt the above-mentioned results in such a way that they can straightforwardly be transferred to randomized Krylov methods. We further show that the Krylov approximation error norm can be rigorously bounded in terms of the sketched residual norm, demonstrating that the norm of the sketched residual can be used as a reliable, non-heuristic stopping criterion.

In addition to the contributions mentioned above, we discuss some possible refinements in the implementation of sketched Krylov methods that can improve their efficiency and numerical stability.

The remainder of this paper is organized as follows. In \Cref{sec:basics} we introduce basic material on Krylov subspace methods for matrix functions, before specifically focusing on problems with an underlying ODE. Based on this, we introduce a general framework for Krylov-based ODE residuals in \Cref{sec:framework}, which unifies existing results from the literature and allows us to obtain conceptually simpler proofs for important properties. We specifically apply this framework to sketched Krylov methods in \Cref{sec:sketching_ode} and illustrate how it can be used to obtain a rigorous a posteriori error estimate for these methods. In this section, we also discuss the efficient implementation of residual-based sketched Krylov methods, focusing in particular on the possibility of introducing restarts in order to improve numerical stability and efficiency. Numerical experiments on several large-scale real-world problems are presented in \Cref{sec:experiments}. Concluding remarks and an outlook on potential future research topics are given in \Cref{sec:conclusions}.

\subsection{Notation and conventions}\label{subsec:notation}
Throughout the paper, we use the following notations. By $\langle\cdot,\cdot\rangle$, we denote the Euclidean inner product and by $\|\cdot\|$ we denote the vector or matrix norm it induces. The transpose of a matrix $M$ is denoted by $M\tp$. The spectrum of a matrix $M$ is denoted by $\sigma(M)$. The imaginary part of a complex number $z$ is denoted as $\Im(z)$. The $m$th canonical unit vector is denoted by $\ve_m$. The $i$th entry of a vector $\vv$ is denoted by $[\vv]_i$, and when indexing vectors or matrices, we sometimes use ``MATLAB-style notation'', i.e., $i:j$ refers to all indices $i,i+1,\dots,j-1,j$.
In the context of block Krylov methods (with block size $\ell$), we denote ``block vectors'' $\pU \in \Rg{n}{\ell}$ by uppercase calligraphic font and ``block scalars'' $\pg \in \Rsq{\ell}$ by lowercase calligraphic font.

For ease of exposition, we restrict to real-valued matrices and vectors throughout the manuscript, although most results are rather straightforwardly transferable to the complex-valued case.

\section{Basic Material}\label{sec:basics}
In this section, we present basic material on matrix functions and Krylov subspace methods, as well as their application in the context of solving ODEs.

\subsection{Functions of matrices}\label{subsec:matfun}
Matrix functions can be defined in many different ways. The most common definitions are based on the Jordan canonical form, Hermite interpolating polynomials, or the Cauchy integral formula; we refer to~\cite[Section~1.2]{Higham2008} for a thorough treatment. 

As we are mostly interested in entire functions in this work, we only recall the definition based on the Cauchy integral formula, which applies to functions $f$ that are analytic in a region $\Omega \subseteq \C$ that contains the spectrum $\sigma(A)$ of $A \in \Rsq{n}$. We can then define $f(A)$ via
\[
    f(A) := \frac{1}{2\pi i}\int_{\Gamma} f(z)(z I - A)^{-1} \d z,
\]
where $\Gamma \subset \Omega$ is a path that winds around the spectrum of $A$ exactly once.

It is important to note that even when $A$ is very sparse, $f(A)$ is in general a dense matrix, so that for large $n$, it is not feasible to explicitly form $f(A)$, as it is often impossible to store (in addition to being very costly to compute). Therefore, in applications where the action of $f(A)$ on a vector $\vb$ is required---like in exponential integrators---one typically aims to directly approximate $f(A)\vb$ by an iterative method. The most popular such methods are \emph{Krylov subspace methods}, which we brief\/ly introduce next.

\subsection{Krylov methods for matrix functions}\label{sec:krylov}
The backbone of Krylov methods for matrix functions is the \textit{Arnoldi method} \cite{Arnoldi1951}, which computes an orthonormal basis (ONB) of the Krylov subspace $\kryab{m} := \Span\cbr{\vb,A\vb,...,A^{m-1}\vb}$ via a modified Gram--Schmidt approach.

After $m$ iterations, the Arnoldi method yields a matrix $U_m = \bbr{\vu_1,\vu_2,...,\vu_m} \in \Rg{n}{m}$ that contains a nested orthonormal basis of $\kryab{m}$, and an upper Hessenberg matrix $\uG_m \in \Rg{m+1}{m}$ containing the orthogonalization coefficients. These matrices are connected by the \textit{Arnoldi relation}
\begin{equation}\label{eq:arnoldi relation}
    AU_m = U_{m+1}\uG_m = U_mG_m + g_{m+1,m}\vu_{m+1}\ve_m\tp,
\end{equation}
where $G_m \in \Rsq{m}$ is the upper $m \times m$ part of $\uG_m$, i.e.,
\begin{equation*}\label{eq:arnoldi hessenberg structure}
    \uG_m = \begin{bmatrix}
        G_m \\ g_{m+1,m}\ve_m\tp
    \end{bmatrix}.
\end{equation*}
Clearly,~\eqref{eq:arnoldi relation} directly implies 
\begin{equation*}\label{eq:arnoldi identity H_m}
    U_m\tp AU_m = G_m
\end{equation*}
due to the orthogonality of the columns of $U_m$.

Based on these quantities, for a general matrix function $f\br{A}\vb$, the \textit{Arnoldi (or FOM) approximation} is defined as
\begin{equation}\label{eq:fom}
f(A)\vb \approx \vf_m = U_mf\br{U_m\tp A U_m}U_m\tp\vb = U_mf\br{G_m}\ve_1\norm{\vb}{}.
\end{equation}

With $\text{nnz}(A)$ denoting the number of nonzeros in $A$, the Arnoldi method generally has a computational cost of $\OO{\text{nnz}(A)m + nm^2}$ operations, where the first term accounts for the matrix vector products, while the second term accounts for the modified Gram--Schmidt orthogonalization. In many practical situations, e.g., when $A$ results from the finite difference or finite element discretization of a differential operator, $\text{nnz(A)} = \OO{n}$, so that the orthogonalization causes the dominating computational cost. In addition, storing the (dense) ONB requires $\OO{nm}$ memory, which can also be prohibitive if a large number $m$ of iterations is required to reach the desired accuracy.

There exist different theoretical justifications for why~\eqref{eq:fom} is a sensible approximation for $f(A)\vb$. We recapitulate two of the most important and widely used ones here. First, the Arnoldi iteration can alternatively be characterized as
\[
\vf_m = p_{m-1}(A)\vb
\]
where $p_{m-1}$ is the polynomial that interpolates $f$ at the spectrum of $G_m$ (i.e., at the so-called Ritz values) in the Hermite sense; see~\cite[Theorem~3.3]{Saad1992}. Second, the Arnoldi approximation fulfills
\begin{equation}\label{eq:quasiopt}
\|f(A)\vb-\vf_m\| \leq 2(1+\sqrt{2}) \min_{p \in \Pi_{m-1}} \max_{z \in W(A)}\|f(z) - p(z)\|,
\end{equation}
where $\Pi_{m-1}$ denotes the set of all polynomials of degree at most $m-1$ and $W(A)$ is the numerical range (or field of values) of $A$. Inequality~\eqref{eq:quasiopt} easily follows from the Crouzeix--Palencia theorem~\cite{crouzeix2017numerical} and allows us to relate the convergence of the Arnoldi approximation for $f(A)\vb$ to the convergence of polynomial approximations for the scalar function $f$ on $W(A)$; see, e.g.,~\cite{BeckermannReichel2009}.

In the context of solving ODEs, an alternative way of deriving and justifying the Arnoldi approximation~\eqref{eq:fom} arises from imposing a Galerkin condition, as we will illustrate next.

\subsection{Krylov methods for ODEs}\label{sec:residual}
For illustration purposes, in this section we recapitulate the Krylov solution of linear, homogeneous, first-order IVPs 
\begin{equation}\label{eq:simple ivp}
\begin{cases}
\begin{aligned}
\vy'(t) &= -A\vy(t) \text{ on } (0,T],\\
\vy(0) &= \vb_0,
\end{aligned}
\end{cases}
\end{equation}
where $A \in \Rnn$, as this is the simplest model problem fitting into the framework developed in \Cref{sec:framework}. It is well known that the solution of~\eqref{eq:simple ivp} can be written in terms of the matrix exponential as
\begin{equation*}\label{eq:ode_solution_exp}
    \vy(t) = \exp\br{-tA}\vb_0 \text{ for all } t \in [0, T].
\end{equation*}
Assume we want to find an approximate solution $\vy_m$ to~\eqref{eq:simple ivp} in the Krylov subspace $\spK_m(A,\vb_0)$, i.e., $\vy_m: [0,T] \rightarrow \spK_m(A,\vb_0)$. We introduce the corresponding ODE residual
\begin{equation*}\label{eq:res}
    \vr_m(t) := -\vy_m'(t) - A\vy_m(t) \qquad t \in [0, T].
\end{equation*}
Clearly, $\vr_m(t) \equiv \vnull$ would imply that $\vy_m$ solves the IVP~\eqref{eq:simple ivp}.

One possible approach for defining an approximate solution $\vy_m(t)$ is to impose a Galerkin condition, i.e., one demands that
\begin{equation}\label{eq:galerkin}
\vr_m(t) \perp \spK_m(A,\vb_0) \text{ for all } t \in [0, T].
\end{equation}
Using the notation from \Cref{sec:krylov}, in particular letting $U_m$ be an ONB of $\spK_m(A,\vb_0)$, We can write $\vy_m(t) = U_m\vx_m(t)$ for some $\vx_m(t): [0,T] \rightarrow \R^m$. Condition~\eqref{eq:galerkin} then implies
\[
    \vnull = U_m\tp\vr_m(t) = -U_m\tp(U_m\vx_m'(t) - A U_m \vx_m(t)) = -\vx_m'(t) - G_m \vx_m(t),
\]
i.e., $\vx_m$ solves the projected IVP
\begin{equation}\label{eq:simple ivp proj}
\begin{cases}
\begin{aligned}
\vx_m'(t) &= -G_m\vx_m(t) \text{ on } (0,T],\\
\vx_m(0) &= \ve_1\|\vb_0\|,
\end{aligned}
\end{cases}
\end{equation}
where the initial conditions are established by requiring that $\vy_m(0) = \vb_0$, i.e., that the approximation fulfills the initial conditions of the original IVP~\eqref{eq:simple ivp}. The solution of~\eqref{eq:simple ivp proj} is given by
\[
\vx_m(t) = \exp\br{-tG_m}\ve_1\|\vb_0\|,
\]
so that $\vy_m(t) = U_m\exp\br{-tG_m}\ve_1\|\vb_0\|$ is exactly the Arnoldi approximation~\eqref{eq:fom} for $f(z) = \exp(-tz)$, providing another possible justification for why this is a sensible approximation. A very attractive feature of the Arnoldi approximation in the context of ODEs is that the norm of its residual can be used as a reliable and easy-to-compute stopping criterion. It is straightforward to verify that
\begin{equation}\label{eq:residual norm}
    \norm{\vr_m(t)}{} = \abs{g_{m+1,m}\bbr{\vx_m(t)}_m},
\end{equation}
see, e.g.,~\cite{bgh13,bokn20}, which can be readily computed from quantities available from the Arnoldi process. That the residual norm is indeed a suitable stopping criterion for the iteration can be seen as follows: The ODE residual and the \emph{error}
\begin{equation*}\label{eq:error}
    \vxi_m(t) := \vy(t) - \vy_m(t)
\end{equation*}
are related by the semilinear IVP
\begin{equation*}\label{eq:simple ivp error}
\begin{cases}
\begin{aligned}
\vxi_m'(t) &= -A\vxi_m(t) + \vr_m(t) \text{ on } (0,T],\\
\vxi_m(0) &= \vnull.
\end{aligned}
\end{cases}
\end{equation*}
so that, using the variation-of-constants formula, the error can be written as
\begin{equation}\label{eq:variation_of_contants_error}
\vxi_m(t) = \int_0^t \exp(-(t-\theta)A)\vr_m(s) \d\theta.
\end{equation}
Under the standard assumption that there exist constants $C_1 > 0$ and $\omega_1 \in \R$ such that
\begin{equation}\label{eq:bound exp norm}
    \norm{\exp\br{-tA}}{} \leq C_1e^{-t\omega_1} \quad\text{for all } t\geq 0,
\end{equation}
one directly obtains the bound
\begin{equation}\label{eq:bound_error_phi}
    \errnormt \leq C_1 t\varphi_1\br{-t\omega_1}\thetamax \norm{\vr_m\br{\theta}}{}, \quad\text{where}\quad  \varphi_1\br{z} = \frac{e^z-1}{z}
\end{equation}
from~\eqref{eq:variation_of_contants_error}; see~\cite[Lemma~4.1]{bgh13}. Thus, a \emph{uniform} bound on the residual norm on $[0,T]$ directly implies a bound on the \emph{error norm}.

Note that, e.g., when $-A$ is a stable operator with field of values in the left half plane, condition~\eqref{eq:bound exp norm} is satisfied with $\omega_1 = 0$ and $C_1 = 1$. More generally, $\omega_1$ can be taken as the left end point of the field of values of $A$.

Due to the features mentioned above, it is very attractive to work with residual-based methodologies when solving large-scale ODEs using Krylov subspace methods, and there is a large body of work on this topic~\cite{botchev2013block,bgh13,bokn20,bks23,bkt21,celledoni1997krylov}. In the next section, we aim to unify and generalize existing results in a framework that allows us to extend it to other, more general situations, and in particular to Krylov methods based on randomized sketching (cf.\ \Cref{sec:sketching_ode}).

\section{A general framework for Krylov ODE residuals}\label{sec:framework}
In this section, we consider a general semi-linear differential equation of the form
\begin{equation}\label{eq:general_ode}
    \vy^{(p)}(t) = -A\vy(t) + \vw(t) \text{ for } t \in (0,T],
\end{equation}
where $p \in \mathbb{N}$ and $\vw$ is a vector valued function, $\vw(t):(0,T] \rightarrow \R^n$. We assume that we are given the initial conditions 
\[
\vy(0) = \vb_0, \ \vy^\prime(0) = \vb_1, \ \dots, \ \vy^{(p-1)}(0) = \vb_{p-1}
\]
which, together with~\eqref{eq:general_ode}, define an initial value problem.

\subsection{Efficiently computing the Krylov ODE residual and its norm}
There has been a lot of work on obtaining computationally feasible formulas for the residual of Krylov solutions for~\eqref{eq:general_ode}, for many different special cases. E.g., the papers~\cite{bgh13,bokn20} treat the case $p = 1$ and $\vw(t) \equiv \vnull$,~\cite{bkt21,celledoni1997krylov} treat $p=1$ and a constant inhomogeneity $\vw(t) \equiv \vw \neq \vnull$, while~\cite{botchev2013block,bks23} consider $p = 2$. In each case, formulas for the residual are derived by explicit, lengthy computations that need to be adapted and redone from scratch whenever the structure of the ODE or the employed method changes.

In the following, we present a general framework that exposes the principles underlying all approaches mentioned above and therefore allows us to extend them to other differential equations (and, in particular, to other algorithms) with minimal effort. An additional benefit is that the proofs of the results simplify in comparison to those presented in some of the references above.

We assume that we have the subspaces
\[
\V_m \subseteq \V_{m+1} \subseteq \R^n
\]
of dimension $\dim(\V_i) = d_i$, and that we have the inclusion $A\V_m \subseteq \V_{m+1}$. We further assume that $\vb_0,\dots,\vb_{p-1} \in \V_m$ as well as $\vw(t) \in \V_m$ for all $t \in (0,T]$.\footnote{The assumption that $\V_m$ has to include all initial conditions as well as $\vw(t)$ at all time points seems quite restrictive at first. However, for most practically relevant cases, only few vectors are required to actually lie in $\V_{m-1}$. This is discussed later in~\Cref{sec:limitations of the framework}.}

Let $\vy_m: [0,T] \rightarrow \V_m$ be an approximate solution of~\eqref{eq:general_ode} determined by the Galerkin condition
\begin{equation}\label{eq:galerkin_residual}
\vr_m(t) \perp_\ast \V_m \text{ for all } t \in [0,T]
\end{equation}
with respect to some general inner product $\langle \cdot,\cdot\rangle_\ast$ on $\V_{m+1}$, where $\vr_m$ is the residual
\[
\vr_m(t) = -\vy_m^{(p)}(t) - A\vy_m(t) + \vw(t).
\]
In the following, we denote by $U_{m}$ an orthonormal basis of $\V_{m}$ with respect to $\langle \cdot,\cdot\rangle_\ast$ and
assume that $\pU_{m+1} \in \R^{n \times (d_{m+1}-d_{m})}$ is such that $U_{m+1} = [U_m,\pU_{m+1}]$ is an orthonormal basis of $\V_{m+1}$.\footnote{In case $d_i = \dim(\V_i) = i$, as, e.g., for standard Krylov subspaces, we have $d_{m+1}-d_m = 1$. We allow for more general dimensions of the nested subspaces, such that block Krylov approaches and other, more general methods are also captured by the framework.} The adjoint of a matrix $M$ with respect to $\langle \cdot,\cdot\rangle_\ast$ is denoted by $M^\ast$ in the following. Using this notation, we can write
\[
\vy_m(t) = U_m\vx_m(t) \text{ with some } \vx_m:[0,T] \rightarrow \R^m,
\]
and the corresponding residual is thus given by
\begin{equation}\label{eq:residual_x}
\vr_m(t) = -U_m\vx_m^{(p)}(t) - AU_m\vx_m(t) + \vw(t).
\end{equation}

Any method fitting into the above framework can be interpreted as being based on solving a certain projected version of the initial value problem under consideration, as we outline in the following. Left-multiplying~\eqref{eq:residual_x} by $U_m^\ast$ and applying the Galerkin condition~\eqref{eq:galerkin_residual} results in
\[
\vnull = U_m^\ast\vr_m(t) = -\vx_m^{(p)}(t) - U_m^\ast A U_m\vx_m(t) + U_m^\ast\vw(t).
\]
Demanding that the approximate solution $\vy_m$ exactly satisfies the initial conditions posed at $t=0$, i.e., $\vy_m^{(j)}(0) = \vb_j, j = 0,\dots,p-1$, further yields
\[
U_m\vx_m^{(j)}(0) = \vb_j, \qquad j = 0,\dots,p-1.
\]
Again left-multiplying by $U_m^\ast$ gives the projected initial conditions $\vx_m^{(j)}(0) = U_m\her \vb_j$. In summary, we obtain the projected IVP
\begin{equation}\label{eq:projected_ode}
\begin{cases}
\begin{aligned}
    \vx_m^{(p)}(t) &= -U_m\her A U_m\vx_m(t) + U_m\her\vw(t) \text{ on }(0,T], \\
    \vx_m^{(j)}(0) &= U_m\her \vb_j, \ j = 0,\dots,p - 1,
\end{aligned}
\end{cases}
\end{equation}
that determines the function $\vx_m(t)$. Note that in case of the standard Arnoldi method applied to a homogeneous first-order ODE, the IVP~\eqref{eq:projected_ode} agrees with~\eqref{eq:simple ivp proj}.

\begin{theorem}\label{thm:residual}
Let
\begin{itemize}
    \item[(i)] $\V_m \subseteq \V_{m+1}$ with nested orthonormal bases $U_m,U_{m+1}$ such that $A\V_m \subseteq \V_{m+1}$,
    \item[(ii)] $\vb_0,\dots,\vb_{p-1} \in \V_m$ as well as $\vw(t) \in \V_m$ for all $t \in (0,T]$,
    \item[(iii)] $\vy_m(t) := U_m\vx_m(t)$ be an approximate solution of \eqref{eq:general_ode} determined by the projected IVP \eqref{eq:projected_ode} obtained via the Galerkin condition \eqref{eq:galerkin_residual}.
\end{itemize}
Then,
\[
\vr_m(t) = \pU_{m+1}\bbeta_m(t) \quad\text{ with } \quad\bbeta_m(t) = -\pU_{m+1}^\ast A U_m\vx_m(t).
\]
If further there is a subspace $\V_{m-1} \subseteq \V_m$ with basis $U_{m-1}$ such that $A\V_{m-1} \subseteq \V_m$ and $U_m = [U_{m-1}, \pU_m]$, the expression for $\bbeta_m(t)$ simplifies to
\[
\bbeta_m(t) = -\pU_{m+1}^\ast A \pU_m[\vx_m(t)]_{d_{m-1}+1:d_m},
\]
where $\dim(\V_i) = d_i, i \in \{m-1, m\}$.
\end{theorem}
\begin{proof}
Due to the assumptions that $A\V_m \subseteq \V_{m+1}$ and $\vw(t) \in \V_m \subseteq \V_{m+1}$, we clearly have that $\vr_m(t) \in \V_{m+1}$, so that $\vr_m(t) = U_{m+1}\vc_{m+1}(t)$ with some time-dependent coefficient vector $\vc_{m+1}(t)$.  The Galerkin condition~\eqref{eq:galerkin_residual} implies
\begin{align*}\label{eq:proj_res}
\vc_{m+1}(t)    &= U_{m+1}^\ast\vr_m(t) = \begin{bmatrix} U_m^\ast\vr_m(t) \\ \pU_{m+1}^\ast\vr_m(t) \end{bmatrix} \\ 
                &= \begin{bmatrix} \vnull \\ \pU_{m+1}^\ast (-U_m\vx_m^{(p)}(t) - AU_m\vx_m(t) + \vw(t)) \end{bmatrix} \\ 
                &= \begin{bmatrix} \vnull \\ -\pU_{m+1}^\ast A U_m\vx_m(t) \end{bmatrix},
\end{align*}
where we have used the fact that the columns of $\pU_{m+1}$ are orthogonal to $U_m$ and $\vw(t)$ in the last equality. This proves the first assertion. 

For the second assertion, note that under the additional assumptions, the columns of $AU_{m-1}$ lie in $\V_m$ and are therefore orthogonal to the columns of $\pU_{m+1}$. This immediately gives
\[
\pU_{m+1}^\ast A U_m = [\pU_{m+1}^\ast AU_{m-1}, \pU_{m+1}^\ast A\pU_m] = [\vnull\her,\pU_{m+1}^\ast A\pU_m],
\]
i.e., $-\pU_{m+1}^\ast A U_m\vx_m(t) = -\pU_{m+1}^\ast A \pU_m[\vx_m(t)]_{d_{m-1}+1:d_m}$.
\end{proof}

Denoting by $\|\cdot\|_\ast$ the norm that is induced by $\langle\cdot,\cdot\rangle_\ast$, \Cref{thm:residual} allows us to easily obtain $\|\vr_m(t)\|_\ast$.

\begin{corollary}\label{cor:resnorm}
Using the notation of \Cref{thm:residual}, we have
\[
\|\vr_m(t)\|_\ast = \|\bbeta_m(t)\|.
\]
\end{corollary}
\begin{proof}
Using the result of \Cref{thm:residual}, the assertion follows from the straightforward computation
\begin{align*}
\|\vr_m(t)\|_\ast^2 &= \langle \vr_m(t), \vr_m(t)\rangle_\ast = \left\langle \pU_{m+1}\bbeta_m(t), \pU_{m+1}\bbeta_m(t) \right\rangle_\ast \\
&= \langle\bbeta_m(t), \bbeta_m(t) \rangle = \|\bbeta_m(t)\|^2.
\end{align*}
\end{proof}

We now first discuss which inhomogeneities $\vw(t)$ are naturally covered by our framework before providing several examples for its application in~\Cref{sec:examples}.

\subsection{Non-constant inhomogeneities}\label{sec:limitations of the framework}
For arbitrary vector-valued in\-homo\-gen\-ei\-ties $\vw(t)$, the requirement that $\vw(t) \in \V_m$ for all $t$ in \Cref{thm:residual} might be very restrictive and may not be satisfied by typical approximation spaces. We are, however, primarily interested in situations in which the solution of the IVP can be expressed in terms of certain matrix functions in order to then apply (rational) Krylov methods. In these situations, we typically have, for a moderate value of $q$,
\begin{equation}\label{eq:inhom linear combination}
\vw(t) = \sum_{i=1}^q g_i(t)\vw_i,
\end{equation}
with $g_i:(0,T] \rightarrow \R, i = 1,\dots,q$, i.e., $\vw(t)$ is a linear combination (with time-dependent coefficients) of only a few constant vectors. In this case, the solution to the IVP can often be written in the form
\begin{equation}\label{eq:y_bw}
\vy(t) = \sum_{i=0}^{p-1}f^{\text{IC}}_i(t,A)\vb_i + \sum_{j=1}^qf^{\text{S}}_{j}(t,A)\vw_j,
\end{equation}
with time-dependent matrix functions $f^{\text{IC}}_i,f^{\text{S}}_j:[0,T]\times\Rsq{n}\rightarrow\Rsq{n}, i = 0,\dots,p-1, j=1,\dots,q$.

The arguably most important special case of inhomogeneities of the form~\eqref{eq:inhom linear combination} arises in initial value problems defining the so-called $\varphi$-functions
\[
\varphi_j(z) = \sum_{i=0}^\infty \frac{z^i}{(i+j)!}
\]
that form the basis of most exponential integration schemes~\cite{HochbruckOstermann2010,NiesenWritght2012}: The IVP
\begin{equation}\label{eq:phi_ivp}
    \begin{cases}
        \begin{aligned}
            \vy^\prime(t) &= -A\vy(t) + \sum_{j=1}^q\frac{t^{j-1}}{(j-1)!}\vw_j \text{ on }(0,T], \\
            \vy(0) &= \vb_0,
        \end{aligned}
    \end{cases}
\end{equation}
is solved by
\begin{equation}\label{eq:general exponential integrator}
    \vy(t) = \exp(-tA)\vb_0 + t\varphi_1(-tA)\vw_1+t^2\varphi_2(-tA)\vw_2 + \cdots + t^q\varphi_q(-tA)\vw_q.
\end{equation}

In principle,~\eqref{eq:y_bw} could, e.g., be approximated using a block Krylov method with starting block $\pB = [\vb_0,...,\vb_{p-1},\vw_1,...,\vw_q]$ (cf.~\Cref{example:block_krylov}). 
Often, however, it is beneficial to first rewrite the IVP  by splitting $\vy(t) = \widetilde{\vy}(t) + \va(t)$ such that $\widetilde{\vy}^{(j)} = \vnull, j = 0,\dots,p-1$, leading to
\[
\begin{cases}
    \begin{aligned}
        \widetilde{\vy}^{(p)}(t) &= -A\widetilde{\vy}(t) + \vw(t) - A\va(t) - \va^{(p)}(t) \text{ on }(0,T], \\
        \widetilde{\vy}(0) &= \widetilde{\vy}^\prime(0) = \dots = \widetilde{\vy}^{(p-1)}(0) = \vnull.
    \end{aligned}
\end{cases}
\]

E.g., for~\eqref{eq:phi_ivp}, we can set $\va(t) := \sum_{j=0}^{q-1}\frac{t^j}{j!}\widetilde{\vw}_j$ with the recursively defined vectors
\[
\widetilde{\vw}_0 = \vb_0, \qquad \widetilde{\vw}_j = -A\widetilde{\vw}_{j-1} + \vw_j, \ j = 1,\dots,q;
\]
cf.~\cite{NiesenWritght2012}, which leads to the IVP
\begin{equation}\label{eq:phi_ivp_transformed}
\begin{cases}
    \begin{aligned}
        \widetilde{\vy}^\prime(t) &= -A\widetilde{\vy}(t) + \frac{t^{q-1}}{(q-1)!}\widetilde{\vw}_q\text{ on }(0,T], \\ \widetilde{\vy}(0) &= \vnull,
    \end{aligned}
\end{cases}
\end{equation}
The solution of~\eqref{eq:phi_ivp_transformed} is given by $\widetilde{\vy}(t) = t^q\varphi_q(-tA)\widetilde{\vw}_q$ and can thus be approximated in the Krylov space $\kry{m}{A,\widetilde{\vw}_q}$.

\subsection{Examples} \label{sec:examples}
\Cref{thm:residual} and \Cref{cor:resnorm} can be applied under rather general conditions and allow to reproduce several known results as well as to obtain generalizations to many other important situations, as we illustrate in the following examples.

\begin{example}\label{example:exp}
Consider the standard polynomial Arnoldi method for the first-order IVP
\[
\begin{cases}
\begin{aligned}
\vy'(t) &= -A\vy(t) \text{ on } (0,T],\\
\vy(0) &= \vb_0.
\end{aligned}
\end{cases}
\]
that we considered in \Cref{sec:residual}. Here, $\V_m = \spK_m(A,\vb_0)$, the nested orthonormal basis is $U_{m+1} = [\vu_1,\dots,\vu_{m+1}]$, and the Euclidean inner product is used, i.e., $\langle\cdot,\cdot\rangle_\ast=\langle\cdot,\cdot\rangle$. In this case, the projected IVP arising from the Galerkin condition~\eqref{eq:galerkin_residual} is~\eqref{eq:simple ivp proj} and  $\vy_m(t) = U_m\exp(-tG_m)\ve_1\|\vb_0\|$ is exactly the standard Arnoldi approximation for the exponential. Due to $\vu_{m+1}^\ast A\vu_{m} = g_{m+1,m}$, \Cref{thm:residual} and \Cref{cor:resnorm}  directly yield the well-known result~\eqref{eq:residual norm}.
\end{example}

\begin{example}\label{example:exp_Q_ip}
Our framework allows us to effortlessly obtain the exact same results as in \Cref{example:exp} also when a non-standard inner product is used. In certain applications, switching to a different inner product allows to measure the error in a more natural way and can additionally have algorithmic advantages. As an example, consider the Poisson system
\begin{equation}\label{eq:poisson_system}
\begin{cases}
\begin{aligned}
\vy^\prime(t) &= JQ\vy(t) \text{ on }(0,T],\\
\vy(0) &= \vb_0.
\end{aligned}
\end{cases}
\end{equation}
with quadratic Hamiltonian $\mathcal{H}(\vy) = \frac12 \vy\tp Q\vy$, where $Q$ is symmetric positive definite and the structure matrix $J$ is skew-symmetric, i.e., $J\tp = -J$. For such a problem, it is natural to measure the error (or residual) norm using the inner product $\langle \vv, \vw \rangle_Q = \langle Q\vv, \vw \rangle$ induced by $Q$. In addition, the so-called $Q$-Arnoldi method (i.e., the Arnoldi method using $\langle\cdot,\cdot\rangle_Q$ instead of the Euclidean inner product) has two further important advantages; see, e.g.,~\cite{maier2025energy}. First, the system matrix $JQ$ is skew-self-adjoint with respect to the $Q$-inner product so that the Krylov basis vectors fulfill a three-term recurrence akin to the Lanczos process. Second, the corresponding Krylov approximations
\begin{equation}\label{eq:poisson_system_approx}
\vy_m(t) := U_m\exp(tG_m)\ve_1\|\vb_0\|_Q
\end{equation}
are energy-preserving---i.e., $\mathcal{H}(\vy_m(t)) = \mathcal{H}(\vy_m(0))$ for all $t > 0$---as it is also the case for the exact solution of the Poisson system~\eqref{eq:poisson_system}. Analogously to \Cref{example:exp}, we obtain 
\[
\|\vr_m(t)\|_Q = |g_{m+1,m}|\|\vb_0\|_Q|[\exp(tG_m)]_{m,1}|,
\]
where $G_m$ is the tridiagonal matrix arising from the $Q$-Arnoldi process. Such a relation is certainly of value: In~\cite[Section~3.3]{maier2025energy}, the authors mention the lack of a natural concept for monitoring the progress of the method as a main reason for not using~\eqref{eq:poisson_system_approx} in practice and resorting to an approach based on solving linear systems instead.
\end{example}

\begin{example}\label{example:block_krylov}
The solution of the second-order IVP
\begin{equation*}\label{eq:ivp_2nd_order}
\begin{cases}
\begin{aligned}
\vy^{\prime\prime}(t) &= -A\vy(t) + \vw  \text{ on } (0,T],\\
\vy(0) &= \vnull, \ \ \vy^\prime(0) = \vb_1
\end{aligned}
\end{cases}
\end{equation*}
is given by
\begin{equation}\label{eq:solution_2nd_order}
    \vy(t) = A^{-1}\left(I-\cos\big(t\sqrt{A}\big)\right)\vw + \big(\sqrt{A}\big)^{-1}\sin\big(t\sqrt{A}\big)\vb_1.
\end{equation}
The approach proposed in~\cite{botchev2013block} for approximating~\eqref{eq:solution_2nd_order} is to use the block Krylov subspace
\begin{equation*}\label{eq:block_krylovspace}
\spK_m^\Box(A,\pB) = \text{colspan}\left(\pB,A\pB,A^2\pB,\dots,A^{m-1}\pB\right).
\end{equation*}
with $\pB = [\vw, \vb_1]$. Just as for the standard Arnoldi method, we denote the orthonormal basis of $\spK_m^\Box(A,\pB)$ computed by the block Arnoldi method by $U_m \in \mathbb{R}^{n \times 2m}$ and the compression of $A$ onto the block Krylov space as $G_m = U_m\tp AU_m \in \mathbb{R}^{2m \times 2m}$ and impose a Galerkin condition $\vr_m(t) \perp \spK_m^\Box(A,\pB)$ for the residual $\vr_m(t) = -\vy_m^{\prime\prime}(t) - A\vy_m(t) + \vw$. The projected IVP~\eqref{eq:projected_ode} resulting from this condition reads
\[
\begin{cases}
\begin{aligned}
\vx_m^{\prime\prime}(t) &= -G_m\vx_m(t) + U_m\tp\vw \text{ on } (0,T],\\
\vx_m(0) &= \vnull,\ \ \vx_m^{\prime} = U_m\tp\vb_1,
\end{aligned}
\end{cases}
\]
and is solved by
\begin{equation*}\label{eq:approx_2nd_order_x}
\vx_m(t) = \big(G_m\big)^{-1}\left(I-\cos\left(t\sqrt{G_m}\right)\right)U_m\tp\vw + \left(\sqrt{G_m}\right)^{-1}\sin\left(t\sqrt{G_m}\right)U_m\tp\vb_1.
\end{equation*}
The overall block Arnoldi approximation is then given by $\vy_m(t) = U_m\vx_m(t)$.

We can now apply our framework, because the block Krylov subspaces satisfy the nestedness property $A\spK_m^\Box(A,\pB) \subseteq \spK_{m+1}^\Box(A,\pB)$. This yields
\[
\vr_m(t) = -\pU_{m+1}\pg_{m+1,m}[\vx_m(t)]_{2m-1:2m},
\]
where $\pg_{m+1,m} \in \R^{2 \times 2}$ denotes the bottom right block of the block upper Hessenberg matrix arising from the block Arnoldi process. This essentially reproduces the first part of~\cite[Theorem~3]{botchev2013block}.
\end{example}

\begin{example}\label{example:rational_krylov}
In situations where shifted linear systems with $A$ can be efficiently solved, it is often beneficial to work with rational Krylov spaces instead of standard (polynomial) Krylov spaces due to the better approximation properties of rational functions compared to polynomials; see, e.g.,~\cite{Guettel2010,Guettel2013,GuettelKnizhnerman2013}. A rational Krylov subspace of dimension $m$ with respect to $A$ and $\vb$ is defined as
\begin{equation*}\label{eq:rational_krylov_space}
    \spQ_m(A,\vb,q_{m-1}) := q_{m-1}(A)^{-1}\spK_m(A,\vb),
\end{equation*}
with a denominator polynomial $q_{m-1}(z) = \prod_{i=1}^{m-1} (z-z_i)$ with zeros $z_i \notin \sigma(A)$, the so-called \emph{poles} of the rational Krylov space.

Even for nested pole sequences, i.e., $q_m(z) = (z-z_m)q_{m-1}(z)$, we generally have $A\spQ_m(A,\vb,q_{m-1}) \not\subset \spQ_{m+1}(A,\vb,q_m)$, so that \Cref{thm:residual} cannot directly be applied. However, if the pole $z_m$ is chosen ``at infinity'', i.e., $q_{m-1}(z) = q_{m}(z)$, the inclusion $A\spQ_{m}(A,\vb,q_{m-1}) \subset \spQ_{m+1}(A,\vb,q_{m-1}) = \spQ_{m+1}(A,\vb,q_m)$ \emph{does} hold. 

This fits quite naturally into the rational Krylov framework. Instead of the Arnoldi relation~\eqref{eq:arnoldi relation}, rational Krylov methods yield a relation
\begin{equation}\label{eq:rational_arnoldi_relation}
AU_{m+1}\uK_m = U_{m+1}\uG_m
\end{equation}
with unreduced upper Hessenberg matrices $\uG_m,\uK_m$ of size $(m+1) \times m$. If $z_m = \infty$, the last row of $\uK_m$ vanishes and the relation~\eqref{eq:rational_arnoldi_relation} simplifies to
\[
AU_mK_m = U_{m+1}\uG_m,
\]
where $K_m$ denotes the upper $m \times m$ part of $\uK_m$. In this case, the projection of $A$ onto the rational Krylov space can be computed as
\[
U_m\tp A U_m = G_mK_m^{-1}
\]
working only with matrices of size $m \times m$; see, e.g.,~\cite[Section~3.1]{Guettel2013}. Therefore, in order to efficiently evaluate the rational Arnoldi approximation to $f(A)\vb$, it is common practice in rational Krylov methods to first do a ``polynomial step'', then form the rational Arnoldi approximation
\[
\vf_m^{\text{rat}} := U_mf(G_mK_m^{-1})\ve_1\|\vb\|
\]
and only after that ``complete'' the rational Arnoldi step by adding a new pole $z_m \neq \infty$ (if the iteration needs to be continued). 

Exactly the same approach can be used for computing the rational Krylov residual whenever~\eqref{eq:rational_arnoldi_approx} is used in the context of solving an ODE. For ease of exposition, we describe the approach for the simplest case, i.e., for solving IVP~\eqref{eq:simple ivp}, noting that it is also straightforwardly applicable for all other IVPs. The corresponding rational Arnoldi approximation reads
\begin{equation}\label{eq:rational_arnoldi_approx}
\vy_m^{\text{rat}}(t) = U_m\exp(-tG_mK_m^{-1})\ve_1\|\vb_0\|.
\end{equation}

It is straightforward to verify that the rational Arnoldi approximation~\eqref{eq:rational_arnoldi_approx} is characterized by the Galerkin condition
\[
\vr_m^{\text{rat}}(t) \perp \spQ_m(A,\vb,q_{m-1})
\]
for the residual $\vr_m^{\text{rat}}(t) := -(\vy_m^{\text{rat}})^{\prime}(t) - A\vy_m^{\text{rat}}(t)$. Thus, from \Cref{thm:residual} and \Cref{cor:resnorm}, we obtain the relation
\[
\vr^{\text{rat}}_m(t) = \beta_m(t)\vu_{m+1} \text{ with } \beta_m(t) = -\vh_{m+1,\bullet}\exp(-tG_mK_m^{-1})\ve_1\|\vb_0\|,
\]
where $\vh_{m+1,\bullet}$ is the last row of the matrix 
\[
\uH_m := \uG_m K_m^{-1}.
\]
We note that further simplifications are not possible because the assumptions for the second assertion of \Cref{thm:residual} are not fulfilled in general for rational Krylov spaces, unless two consecutive poles are chosen at infinity.
\end{example}

\subsection{Bounding the error in terms of the residual}
Similar to what we discussed in \Cref{sec:residual}, the residual norm can be used to obtain a bound for the \emph{error} of the Krylov ODE solution. We note that the Krylov error $\vxi_m(t) = \vy(t) - \vy_m(t)$ satisfies
\begin{align*}
    \vxi_m^{(p)}(t) &= \vy^{(p)}(t) - \vy_m^{(p)}(t) \\
    &= -A\vy(t) + \vw(t) - \vy_m^{(p)}(t) \\
    &= -A\vy(t) + \vw(t) - \vy_m^{(p)}(t) + A\vy_m(t) - A\vy_m(t) \\
    &= -A(\vy(t) - \vy_m(t)) + \left(-\vy_m^{(p)}(t) - A\vy_m(t) + \vw(t)\right) \\
    &= -A\vxi_m(t) + \rmt.
\end{align*}
Therefore, $\vxi_m(t)$ can be characterized as the solution of the IVP
\begin{equation}\label{eq:error ivp}
\begin{cases}
    \begin{aligned}
        \vxi_m^{(p)}(t) &= -A\vxi_m(t) + \vr_m(t) \text{ on }(0,T], \\
        \vxi_m(0) &= \vxi_m^\prime(0) = \cdots = \vxi_m^{(p-1)}(0) = \vnull.
    \end{aligned}
\end{cases}
\end{equation}
The solution of~\eqref{eq:error ivp} can be written in terms of a (generalized) variation-of-constants formula
\begin{equation}\label{eq:generalized_variation_of_constants}
\vxi_m(t) = \int_0^tK_p\br{t-\theta,A}\vr_m\br{\theta} \d\theta,
\end{equation}
where the kernel function $K_p(\tau,A)$ is the inverse Laplace transform (w.r.t.\ the variable~$s$) of the function $F(s, A) = (s^p I+A)^{-1}$; see, e.g.,~\cite[Chapter~3]{ArendtBattyHieberNeubrander2011}. It follows from~\cite[Equation~(59)]{van2020mittag} that 
\[
K_p(\tau,A) = \tau^{p-1}E_{p,p}(-\tau^p A),
\]
where $E_{\alpha,\beta}(z) = \sum\limits_{j=0}^\infty \frac{z^j}{\Gamma(\alpha j+\beta)}$ is the two-parameter Mittag--Leffler function and $\Gamma(z)$ is the Euler gamma function.

Equation~\eqref{eq:generalized_variation_of_constants} directly implies the error bound
\begin{align}\label{eq:general error bound}
\|\vxi_m(t)\|_\ast &\leq \int_0^t \|K_p\br{t-\theta,A}\|_\ast \|\vr_m\br{\theta}\|_\ast \d\theta \nonumber\\
&\leq \thetamax \|\vr_m\br{\theta}\|_\ast \int_0^t \|K_p\br{t-\theta,A}\|_\ast \d\theta.
\end{align}
More concrete error bounds can be obtained for specific combinations of $K_p$ and $\|\cdot\|_\ast$ by inserting suitable upper bounds for $K_p(\tau, A)$. In the following example, we focus on the two most important special cases, $p = 1$ and $p = 2$.

\begin{example}\label{ex:error_bound_euclidean}
In the first order case we have $K_1(\tau,A) = \exp(-\tau A)$ and for the Euclidean norm, one can again bound $\|K_1(\tau,A)\| \leq C_1 e^{-\tau\omega_1}$ for $\tau \in [0,t]$ to obtain
\begin{equation}\label{eq:bound_error_phi2}
\|\vxi_m(t)\| \leq  \thetamax \|\vr_m\br{\theta}\| \cdot C_1e^{-t\omega_1}\int_0^t e^{-\theta\omega_1} \d\theta = C_1 t\varphi_1(-t\omega_1)\thetamax \|\vr_m\br{\theta}\|,
\end{equation}
exactly reproducing~\eqref{eq:bound_error_phi}. An illustration of this error bound is given in \Cref{fig:errors bounds ress} below. 


In the second order case we have $K_2(\tau,A) = \big(\sqrt{A}\big)^{-1}\sin\big(\tau \sqrt{A}\big)$ and a natural assumption is that $A$ generates an \emph{exponentially bounded cosine family} (see, e.g.,~\cite{fattorini2011second}), i.e.,
\[
\|\cos\big(\tau\sqrt{A}\big)\| \leq C_2 e^{\tau \omega_2},
\]
which implies the bound
\[
\|\big(\sqrt{A}\big)^{-1}\sin\big(\tau \sqrt{A}\big)\| \leq C_2\tau\varphi_1(\tau\omega_2),
\]
so that~\eqref{eq:general error bound} yields
\[
\|\vxi_m(t)\| \leq  \thetamax \|\vr_m\br{\theta}\| \cdot C_2\int_0^t (t-\theta)\varphi_1((t-\theta)\omega_2) \d\theta = C_2t^2\varphi_2(\omega_2 t)\thetamax \|\vr_m\br{\theta}\| 
\]
with the second phi function $\varphi_2(z) = \frac{e^z-1-z}{z^2}$. Specific choices for $C_2$ and $\omega_2$ again depend on properties of $A$. E.g., when $A$ is Hermitian positive definite, one can take $C_2 = 1$ and $\omega_2 = 0$. For non-Hermitian $A$, one can either take $C_2 = 1, \omega_2 = \sqrt{\|A\|}$ or $C_2 = \kappa(X), \omega_2 = \max_{\lambda \in \sigma(A)} \Im\big(\sqrt{\lambda}\big)$, where $\kappa(X)$ is the condition number of an eigenvector matrix of $A$.
\end{example}

\begin{example}\label{example:exp_Q_ip_error}
We revisit the Poisson system~\eqref{eq:poisson_system} considered in \Cref{example:exp_Q_ip}. Recall that its solution is given by $\exp(tJQ)\vb_0$, where $Q$ is symmetric positive definite and $J$ is skew-symmetric. Thus, $\exp(tJQ)$ is unitary with respect to the $Q$-inner product $\langle\cdot,\cdot\rangle_Q$, which immediately yields $\|\exp(tJQ)\|_Q = 1$ for all $t \geq 0$. We thus directly find the error bound
\[
\|\vxi_m(t)\|_Q \leq  t\cdot \thetamax \|\vr_m\br{\theta}\|_Q,
\]
for the approximation produced by the $Q$-Arnoldi method.
\end{example}

\section{Residual-based sketched Krylov methods for ODEs}\label{sec:sketching_ode}
In this section, we discuss in more detail how so-called \emph{sketched Krylov methods}, in particular the sketched Arnoldi (or sFOM) method introduced in~\cite{GuettelSchweitzer2023}, fit into the framework of \Cref{sec:framework} and how this enables us to formulate a reliable stopping criterion for the sketched Arnoldi method. We start by brief\/ly recapitulating some important facts about randomized subspace embeddings, which lie at the heart of sketched Krylov methods.

\subsection{Randomized sketching}\label{sec:sketching}
Sketched Krylov methods (and more generally all methods employing the sketching paradigm) are based on \emph{subspace embeddings}~\cite{DrineasMahoneyMuthukrishnan2006,MartinssonTropp2020,sarlos2006improved,woodruff2014sketching}, which embed a subspace $\V$ of $\R^n$ into a smaller Euclidean space $\R^s$, $s \ll n$ in a way that distorts norms and inner products in a controlled manner. 

For a given subspace $\V$ and distortion parameter $\varepsilon\in [0,1)$, we call $S \in \R^{s \times n}$ an \emph{$\varepsilon$-subspace embedding for $\V$}, if for all $\vv \in \V$,
\begin{equation}\label{eq:sketch}
(1-\varepsilon) \| \vv \|^2 \leq \| S \vv\|^2 \leq (1+\varepsilon) \|\vv\|^2,
\end{equation}
or equivalently, for all $\vu,\vv \in \V$,
\begin{equation}{\label{eq:sketch_innerproduct}}
\langle \vu, \vv \rangle - \varepsilon \| \vu \| \| \vv \|
            \leq \langle S \vu, S \vv \rangle 
            \leq \langle \vu, \vv \rangle + \varepsilon \| \vu \| \| \vv \|.
\end{equation}
Often, we simply use the established short-hand name \emph{``sketching matrix''} for a subspace embedding $S$. 

In practice, one typically does not have complete knowledge of the subspace $\V$ that shall be embedded. E.g., in the context of the methods discussed in this paper, $\V = \spK_{m+1}^\Box(A,\pB)$ is a (block) Krylov subspace that has not yet been constructed at the time at which the sketching matrix needs to be fixed. 

Therefore, we work with so-called \emph{oblivious embeddings}: For a given subspace dimension $d$, distortion parameter $\varepsilon$ and failure probability $\delta$, we call $\mathcal{S} \subset \R^{s \times n}$ a family of $(\varepsilon,\delta,d)$-oblivious subspace embeddings if for any subspace $\V$ of dimension $d$, a randomly drawn $S \in \mathcal{S}$ satisfies property~\eqref{eq:sketch} and~\eqref{eq:sketch_innerproduct} with probability $1-\delta$.

There are many different ways to probabilistically construct families of oblivious subspace embeddings, and we refer the reader to \cite{woodruff2014sketching} or~\cite[Section~8]{MartinssonTropp2020} for details. In our experiments reported in \Cref{sec:experiments}, we always employ a so-called \emph{sparse sign matrix} as sketching operator, i.e., a matrix of the form
\[
S = \sqrt{\frac{n}{s}} \cdot [\vs_1,\dots,\vs_n] \in \R^{s \times n},
\]
where each column $\vs_i$ contains exactly $\zeta \in \N$ nonzero entries (at randomly chosen positions) which take the values $-1$ or $1$ with equal probability. Clearly, computing the sketch $S\vv$ for some vector $\vv\in\Rn$ is possible in $\mathcal{O}(\zeta n)$ flops. The analysis in~\cite{cohen2016nearly} shows that sparse sign matrices form a family of $(\varepsilon,\delta,d)$-oblivious subspace embeddings if the embedding dimension and sparsity parameter satisfy $s = \mathcal{O}(\varepsilon^{-2}d\log\frac{d}{\delta})$ and $\zeta = \mathcal{O}(\varepsilon^{-1}d \log\frac{d}{\delta})$, respectively. Empirically, even smaller sparsity parameters appear to work very well in practice. E.g., in~\cite[Section~3.3]{tropp2019streaming}, $\zeta = \min\{s,8\}$ is empirically found to be a very reliable choice, and it is mentioned that there is evidence that even smaller $\zeta$ is often feasible, in particular for large problems. In our experiments reported in \Cref{sec:experiments}, we go as low as $\zeta = 1$ without observing deterioration of the algorithms.

The quadratic dependence of the sketching dimension on the embedding accuracy, $s = \mathcal{O}(\varepsilon^{-2}d \log\frac{d}{\delta})$, means that one typically needs to accept rather crude embedding qualities $\varepsilon$ in order to arrive at a computationally feasible method. A typical choice in many algorithms is to aim for $\varepsilon = 1/\sqrt{2}$, so that $s \sim 2d\log d$.

A viewpoint on the sketching paradigm that is very important in our derivations is that a subspace embedding $S$ induces a (semidefinite) inner product
\begin{equation}\label{eq:semidef_inner_product}
    \langle \vu, \vv\rangle_S := \langle S\vu, S\vv\rangle
\end{equation}
and a corresponding semi-norm $\|\vv\|_S =  \sqrt{\langle\vv,\vv\rangle_S}$. It is easy to see that, due to the embedding property~\eqref{eq:sketch_innerproduct}, $\langle \cdot,\cdot\rangle_S$ is a proper (i.e., positive definite) inner product when restricted to the embedded subspace $\V$; see, e.g.,~\cite[Section~3.1]{BalabanovNouy2019}. Consequently, $\|\cdot\|_S$ is a norm on $\V$. In view of~\eqref{eq:sketch_innerproduct}, one can view $\langle \cdot,\cdot\rangle_S$ and $\|\cdot\|_S$ as slightly distorted versions of the Euclidean inner product and norm, respectively.

\subsection{Applying our framework to the sketched Arnoldi method}
We now apply our framework to a sketched Arnoldi method, in particular the implementation of this method employing ``basis whitening'' that is discussed and analyzed in~\cite{GuettelSchweitzer2023,palitta2025sketched}.
In a nutshell, the idea of this method is to cheaply construct a non-orthogonal basis of a Krylov subspace and then retrospectively and implicitly orthogonalize it with respect to the sketched inner product~\eqref{eq:semidef_inner_product}.

In the following, we denote by $V_{m+1} = [\pV_1,\dots,\pV_{m+1}]$ an arbitrary (typically non-orthogonal) nested basis of the (block) Krylov subspace $\spK_{m+1}^\Box(A,\pB)$, where $\pB \in \Rg{n}{\ell}$ and we assume that $\spK_{1}^\Box(A,\pB) = \colspan\{\pV_1\}$ contains all initial conditions, as well as $\vw(t)$ for all $t \in (0,T]$. In the sketched Arnoldi method, we typically obtain $V_{m+1}$ by a $k$-truncated Arnoldi process which also yields the (block-)$k$-banded upper Hessenberg matrix $\uH_m$. In the so-called basis whitening step, we compute a thin QR decomposition $SV_{m+1} = Q_{m+1}R_{m+1}$,
\[
Q_{m+1} = \begin{bmatrix}
    Q_m & \wQ_{m+1}
\end{bmatrix}, \ \ \ R_{m+1} = \begin{bmatrix}
    R_m & \pR_m \\
    0 & \pp_{m+1}
\end{bmatrix}.
\]
Then, $U_{m+1} := V_{m+1}R_{m+1}\inv$ is a basis of $\kry{m+1}{A,\pB}$ which is orthonormal with respect to $\sprod{\cdot}{\cdot}_S$. The \textit{sketched Arnoldi relation} from \cite[Proposition~2.1]{palitta2025mateq} involving these quantities reads
\begin{equation}\label{eq:sketched arnoldi decomposition}
    SAU_{m} = SU_m\br{\Hhat_m + \Chat_m \pE_m\tp} + S\pU_{m+1}\pp_{m+1}\ph_{m+1,m}\pp_m\inv \pE_m\tp,
\end{equation}
where $\Hhat_m = R_mH_mR_m\inv$, $\Chat_m = \pR_m\ph_{m+1,m}\pp_m\inv$, and $\pE_m \in \Rg{m\ell}{\ell}$ contains the last $\ell$ columns of the identity matrix $I_{m\ell}$.


To apply our framework in the $m$th step of the algorithm, we choose
\[
\underbrace{\spK_{m-1}^\Box(A,\pB)}_{=: \V_{m-1}} \subseteq \underbrace{\spK_{m}^\Box(A,\pB)}_{=: \V_m} \subseteq \underbrace{\spK_{m+1}^\Box(A,\pB)}_{=: \V_{m+1}} \subseteq \R^n,
\]
define $\ymt = U_m\xmt$ and demand $\rmt \perp_S \spK_{m}^\Box(A,\pB)$. As $U_{m-1},U_m,U_{m+1}$ are orthogonal with respect to $\sprod{\cdot}{\cdot}_S$ and $A\spK_{j}^\Box(A,\pB) \subseteq \spK_{j+1}^\Box(A,\pB)$ for all $j \in \N$, we have that $\vx_m(t)$ solves the initial value problem from~\eqref{eq:projected_ode} and that we can apply~\Cref{thm:residual} as well as~\Cref{cor:resnorm}. Using~\eqref{eq:sketched arnoldi decomposition} and noting that $M\her = (SM)\tp S$, we obtain
\begin{equation}\label{eq:sketched Hessenberg matrix}
U_m\her AU_m = (SU_m)\tp SAU_m = Q_m^\top SAV_mR_m\inv = \Hhat_m+\Chat_m\pE_m\tp,
\end{equation}
as well as
\begin{equation}\label{eq:resnorm2}
    \| \rhatmt \|_S = \| \betahatmt \|
\end{equation}
with
\begin{equation}\label{eq:resnorm1}
    \begin{split}
        \betahatmt &= -(S\pU_{m+1})\tp SA\pU_m[\xhatmt]_{m(\ell-1)+1:m\ell} \\
        &= -\pp_{m+1}\ph_{m+1,m}\pp_m\inv[\xhatmt]_{m(\ell-1)+1:m\ell}.
    \end{split}   
\end{equation}
Since $\rhatmt \in \spK_{m+1}^\Box(A,\pB)$, we have from the embedding property~\eqref{eq:sketch} that
\begin{equation*}
    \| \rhatmt \| \leq \frac{1}{\sqrt{1 - \varepsilon}} \| \rhatmt \|_S,
\end{equation*}
i.e., if the sketched residual norm---which we can easily monitor during the sketched Arnoldi method due to~\eqref{eq:resnorm2}--\eqref{eq:resnorm1}---is small, this guarantees that the actual residual norm is also small. Plugging~\eqref{eq:resnorm2} into the error bound~\eqref{eq:general error bound} directly yields
\begin{equation}\label{eq:sketched_err_norm}
    \errhatnormt \leq \frac{1}{\sqrt{1-\varepsilon}} \thetamax \|\vr_m\br{\theta}\|_S \int_0^t \|K_p\br{t-\theta,A}\| \d\theta,
\end{equation}
so that, up to a factor $\frac{1}{\sqrt{1-\varepsilon}}$, the same bounds as, e.g., in \Cref{ex:error_bound_euclidean} can be obtained. For example, in the first-order case and under the usual assumption~\eqref{eq:bound exp norm}, the bound~\eqref{eq:sketched_err_norm} turns into
\begin{equation}\label{eq:error_bound_sketched}
\errhatnormt \leq  \frac{C_1}{\sqrt{1-\varepsilon}} t\varphi_1(-t\omega_1)\thetamax \|\vr_m\br{\theta}\|_S,
\end{equation}
so that a small \emph{sketched residual norm} guarantees that the \emph{actual error norm} of the sketched Arnoldi approximation is small as well. One particular thing to note about the bound~\eqref{eq:error_bound_sketched} is that the constants $C_1$ and $\omega_1$ only depend on spectral properties of the matrix $A$, and not on properties of the sketch-projected matrix, whose spectral region is typically much larger than that of $A$. Instead, in our bound~\eqref{eq:error_bound_sketched}, the effect of sketching only enters mildly through the dependence on the embedding quality $\varepsilon$.

We illustrate the bound~\eqref{eq:error_bound_sketched} in \Cref{fig:errors bounds ress} for the model problem considered in \Cref{subsec:conv-diff}, together with the bound obtained for the standard Arnoldi method. We observe that both bounds predict the qualitative behavior of the error very well, although they overestimate the actual magnitude quite substantially---a well-known phenomenon for residual-based error bounds. Both methods and bounds behave almost identical in this experiment. We note for clarity that the bound~\eqref{eq:error_bound_sketched} does \emph{not} say anything about how the errors of the sketched and standard Arnoldi approximations compare, though. 


\begin{figure}
\centering
    \input{fig/err_res_bound_grouped.tikz}
    \caption{FOM and sFOM error and error bound for the convection diffusion problem from \Cref{subsec:conv-diff} for $N=100$. As the spectrum of $A$ goes into the left half plane, we have $\omega_1 \approx -0.5759$ leading to $\varphi(-t\omega_1) \approx 1.3522$. In sFOM, we orthogonalize each new basis vector against the previous $k=2$ vectors and use a sketching dimension of $s = 1000$. As no analytical expression for $\varepsilon$ is available for the embedding we use, we estimate it empirically, giving $\varepsilon\approx 0.1351$.}
    \label{fig:errors bounds ress}
\end{figure}

\subsection{Efficient implementation of the residual-based sketched Arnoldi method} \label{sec:implementation}
We formulate the sketched Arnoldi method in \Cref{alg:residual sfom}. We denote by $\eta_m := \max_{0\leq\theta\leq t} \|\vr_m(\theta)\|_S$ the maximum of the sketched residual norm over the integration interval.

\begin{algorithm}[ht]
    \caption{Residual-based sketched Arnoldi}\label{alg:residual sfom}
    \begin{algorithmic}[1]
        \Input A, $\pB$, $t$, $\mathtt{tol}$, $m_{\max}$
        \Output Approximate solution $\vy_m(t)$
        \Statex\vspace{-0.6\baselineskip}
        \State Compute QR decomposition $Q_{\pB}R_{\pB} = \pB$ and set $V_1 \gets \pB R_\pB\inv$ \label{line:normalizing B}
        \State Perform truncated Arnoldi step $V_1 \mapsto (V_2,\uH_1)$
        \State Compute sketch $S\pV_1$ and QR decomposition $Q_1R_1 = S\pV_1$ and set $\Hhat_1 \gets H_1$
        \For{$m = 1,...,m_{\max}$}
            \State Compute sketch $S\pV_{m+1}$
            \State Update QR decomposition $(Q_mR_m = SV_m) \mapsto (Q_{m+1}R_{m+1} = SV_{m+1})$ \label{line:qr update}
            \If{\texttt{update residual}}
                \State Solve sketch-projected problem \eqref{eq:projected_ode} using \eqref{eq:sketched Hessenberg matrix} to obtain $\xhatmt$
                \State Set $\eta_m \gets \max_{0\leq\theta\leq t} \|\pp_{m+1}\ph_{m+1,m}\pp_m\inv[\vx_m(\theta)]_{m(\ell-1)+1:m\ell}\|$ \label{line:resnorm}
                \If{$\eta_m < \mathtt{tol}$}
                    \State \Return $\vy_m(t) := V_m(R_m\inv\xmt)$
                \EndIf
            \EndIf
            \State Perform truncated Arnoldi step $(V_{m+1},\uH_m) \mapsto (V_{m+2},\uH_{m+1})$
            \State Update $\Hhat_m \mapsto \Hhat_{m+1}$ \label{line:hhat update}
        \EndFor
        \State \Return $\vy_m(t) := V_m(R_m\inv\xmt)$
    \end{algorithmic}
\end{algorithm}

Line \ref{line:normalizing B} corresponds to the ``block normalization'' of $\pB$. If $\pB=\vb$ is a vector, we have $R_\vb = \|\vb\|$ and thus $V_1 = \vb/\|\vb\|$. The update of the QR decomposition in line \ref{line:qr update} of \Cref{alg:residual sfom} can be done by any algorithm that iteratively constructs the QR decomposition column by column. We use modified Gram-Schmidt with reorthogonalization, as it allows us work with a thin QR decomposition. Since the matrix $Q_m$ is not needed for any computation in \Cref{alg:residual sfom}, one might even use a Q-less QR update.

In line \ref{line:resnorm} of \Cref{alg:residual sfom}, $\eta_m$, the maximum of the sketched residual norm, is approximated by computing $\|\vr_m(\theta_i)\|_S$ over an equidistant grid $0 < \theta_1 = t/n_t < ... < \theta_{n_t} = t$, similarly to what is done in~\cite{bokn20,bks23,bkt21}. In all experiments reported in \Cref{sec:experiments}, the residual norm is evaluated at $n_t = 5$ points.

The update of the matrix $\Hhat_m$ in line \ref{line:hhat update} can be done at a cost of $\OO{m^2}$, as described in~\cite[Section 5]{palitta2025mateq}.

\subsection{RT-Restarting for improved stability and efficiency}\label{sec:restarting}
In our own experiments, as well as several experiments reported in the literature (see, e.g.,~\cite{cortinovis2024speeding}), we have observed that for some very hard problems that require many Krylov iterations, the sketched Arnoldi method can become unstable and stagnate at an unsatisfactory accuracy (or even diverge). 

A possible remedy is to \emph{restart} the Arnoldi method once stagnation or divergence is observed (e.g., by monitoring whether the residual norm stops decreasing or starts increasing). In the context of solving ODEs, one can do so via sub time stepping, where the length of the time steps is adaptively determined based on the residual norm. This approach is also known as residual-time (RT) restarting; see~\cite{bokn20,bks23,bkt21}. Based on our formulas for the sketched residual norm, this approach, which we brief\/ly recapitulate in the following, can straightforwardly be generalized to the sketched Arnoldi method.

The crucial observation for this approach is that after an arbitrary number $m$ of Krylov steps and for any tolerance $\mathtt{tol} > 0$, there always exists a time point $\tau > 0$ such that
\[
\max_{0 \leq \theta \leq \tau} \|\vr_m(\theta)\|_S < \mathtt{tol}.
\]
After $m$ steps, if the need for a restart is detected, one can thus start a new sketched Arnoldi iteration for the modified IVP
\begin{equation}\label{eq:rt_restarting_new_ivp}
\begin{cases}
\begin{aligned}
    \widetilde{\vy}^{(p)}(t) &= -A\widetilde{\vy}(t) + \vw(t+\tau) \text{ on }(0,T-\tau], \\
    \widetilde{\vy}^{(j)}(0) &= \vy_m^{(j)}(\tau), \ j = 0,\dots,p-1.
\end{aligned}
\end{cases}
\end{equation}
Clearly, the solution of~\eqref{eq:rt_restarting_new_ivp} is such that $\widetilde{\vy}(T-\tau) = \vy(T)$. Due to the shorter integration interval, problem~\eqref{eq:rt_restarting_new_ivp} is ``easier'' to solve than the original IVP, and each subsequent restart cycle reduces the time interval further.

Our experiments suggest that the sketched Arnoldi method is, in fact, very ``forgiving'' regarding the stabilizing effect of restarts, in the sense that no ``pre-emptive'' restarting is necessary. Even when restarting \emph{after} an instability has been detected, in all our experiments, it was still possible to continue iterating to the desired accuracy after restarting.

In our implementation, we therefore simply monitor the slope at which the residual norm decreases by computing $\slope_m = (\eta_m-\eta_{m-\update})/\update$, where $\update$ is a parameter that determines after how many iterations we compute the residual norm, and $\eta_0 := 0$. Often, the residual norm initially stagnates or \emph{increases} in the first few Krylov iterations before steadily decreasing afterwards. For detecting the need to restart, we therefore first wait until $\slope_m < 0$ for the first time. From then on, we continuously monitor whether $\slope_m$ exceeds $\rtol\cdot\eta_{m-\update}$ with some tolerance $\rtol$, indicating the need for a restart.

For hard problems that require a large number of Krylov iterations in order to reach the desired accuracy, it might also be attractive to restart the sketched Arnoldi method from time to time in order to keep the computational complexity low. While the cost of the sketched Arnoldi method scales much more favorably with respect to the number of iterations than the cost of the plain Arnoldi method, it still \emph{has} an impact. In particular, evaluating a function at the compressed matrix $\Hhat_m + \Chat_m \pE_m\tp$ of size $m\ell$ typically has computational cost $\OO{m^3\ell^3}$. Additionally, a larger number of iterations requires a larger sketching dimension $s$ to guarantee~\eqref{eq:sketch} with high enough probability. We therefore also provide a maximum number $m_{\text{restart}}$ of iterations, after which a restart is performed even when no instability is detected. 

A high-level description of this RT restarting methodology for the sketched Ar\-nol\-di method with block size $\ell$ is given in \Cref{alg:restarting procedure}.
\begin{algorithm}[ht]
    \caption{RT-restarting for sketched Arnoldi}\label{alg:restarting procedure}
    \begin{algorithmic}[1]
        \Input A, $\pB$, $t$, $\mathtt{tol}$, $\mathtt{rtol}$, $m_{\text{restart}}$
        \Output Approximate solution $\vy_m(t)$
        \Statex\vspace{-0.6\baselineskip}
        \State Set $\done \gets \false$
        \While{$\done = \false$}
            \State Run \Cref{alg:residual sfom} with $m_{\max} = m_{\text{restart}}$ or until $\slope_m > \rtol\cdot\eta_{m-\update}$
            \If{$\eta_m < \tol$}
                \State $\done \gets \true$
            \Else
                \If{$\slope_m > \rtol\cdot\eta_{m-\update}$}
                    \State Set $m \gets {m-\update}$
                \EndIf
                \State Compute $\tau$ such that
                    $\max_{0 \leq \theta \leq \tau} \|\vr_m(\theta)\|_S < \mathtt{tol}$
                \State Set $\pB \gets [\ym{\tau},\vy_m^\prime(\tau),\dots,\vy_m^{(p-1)}(\tau)]$ and $t \gets t-\tau$
            \EndIf
        \EndWhile
        \State \Return $\ymt$
    \end{algorithmic}
\end{algorithm}

\section{Numerical experiments} \label{sec:experiments}
In this section, we report numerical experiments on realistic test problems in order to illustrate that the sketched Krylov residual norm yields a reliable stopping criterion for the sketched Arnoldi method. Additionally, we compare the performance of sketched Arnoldi to that of several other established methods, further adding to the existing evidence that it is a very  competitive method for computing the action of matrix functions on vectors.

Specifically, we compare our implementation of residual-based sketched Arnoldi to the residual-time restarting method from~\cite{bokn20,bkt21}\footnote{available at \url{https://team.kiam.ru/botchev/expm/}. We note that an adaptive version of the method is only implemented for the matrix exponential and we thus use the non-adaptive version for better comparability. We have, however, tested both methods for problems involving the exponential and the results were not substantially different.}, to KIOPS (Krylov with Incomplete Orthogonalization Procedure Solver) from~\cite{GaudreaultRainwaterTokman2018}\footnote{available at \url{https://gitlab.com/stephane.gaudreault/kiops}.}---which is based on truncated orthogonalization and sub time stepping driven by a sophisticated error estimate (which is \emph{not} based on the Krylov residual)---as well as the restarted Arnoldi method accelerated by randomized sketching that was recently proposed in~\cite{guidotti2025accelerating}\footnote{A C++ implementation of this method is provided by the authors at~\url{https://gitlab.com/nlg550/randomized-krylov}. However, in order to allow a meaningful comparison, we instead use a MATLAB implementation which is based on the optimized implementation \texttt{funm\_kryl} of the restarted Arnoldi method available at~\url{https://guettel.com/funm_kryl/}, replacing the Arnoldi process by the randomized Arnoldi process from \cite[Algorithm 2.2]{guidotti2025accelerating}.}. For brevity, we denote by
\begin{itemize}
    \item \FOM: the standard Arnoldi method (FOM), i.e.,~\eqref{eq:fom},
    \item \expRT/\phiRT: the residual-time restarting methods from~\cite{bokn20,bkt21} with \\ $m_{\text{restart}} = 30$,
    \item \restartrand: the restarted Arnoldi method accelerated by randomized sketching from~\cite{guidotti2025accelerating},
    \item \texttt{KIOPS250/500}: KIOPS with $m_{\max}=250/500$ from~\cite{GaudreaultRainwaterTokman2018},
    \item \texttt{sFOM250/500}: the sketched Arnoldi method (sFOM) with $m_{\text{restart}}=250/500$.
\end{itemize}

Both KIOPS and sketched Arnoldi can show significantly better performance if a good approximation for the required Krylov dimension is known in advance. Since this requires a priori knowledge, which cannot be expected in practice, we use the maximum Krylov dimensions $m_{\max} = 250,500$ in KIOPS (KIOPS employs a smaller sub time step if the desired accuracy is not reached after $m_{\max}$ iterations), and the restart lengths $m_{\text{restart}} = 250,500$ in sketched Arnoldi across all experiments (although hand-picked values could improve the performance even further).\footnote{Due to the truncated orthogonalization in these algorithms, restart lengths can be chosen much larger than one typically expects, as essentially the only two limiting factors are the memory required for storing the Krylov basis and the time required for evaluating $f$ at the compressed matrix.} For KIOPS, we use the default value $m_{\min} = 10$ for the minimum Krylov dimension.\footnote{This means that the error estimate is first evaluated after 10 iterations. Then, if the desired accuracy is not yet reached, the number of further Krylov iterations to perform until the next evaluation of the error estimate is adaptively determined by the KIOPS algorithm.} As mentioned in \Cref{sec:sketching}, we use sparse sign matrices as subspace embeddings in all experiments, with the sparsity parameter chosen as $\zeta = 1$. The sketching dimension is chosen as $s = 2m_{\text{restart}}$ and a truncation parameter of $k=2$ is used for \sFOM.\footnote{We also ran all experiments using the truncation parameters $k=1, 4$ and $8$. The larger truncation lengths $k = 4,8$ did not lead to a significant improvement in convergence speed or stability, while substantially increasing the computation time due to the larger number of inner products with long vectors. Using $k = 1$ significantly delayed convergence for the problems considered in \Cref{subsec:photonic_crystal,subsec:drum} and showed comparable performance to $k=2$ for the problem from \Cref{subsec:conv-diff}. We therefore do not include detailed results for these runs and always use the truncation parameter $k=2$ in the experiments we report.} For \restartrand, the sketching dimension $s$ and restart length $m_{\text{restart}}$ are chosen as $s = 160$ and $m_{\text{restart}} = 60$ in~\Cref{subsec:conv-diff} and~\Cref{subsec:photonic_crystal}, and as $s = 400$ and $m_{\text{restart}} = 100$ in~\Cref{subsec:drum}, following the choices made by the authors of~\cite{guidotti2025accelerating} for their experiments. In all experiments, \expRT~and \phiRT~use a restart length of $m_{\max} = 30$ iterations.

In our experiments, we track several different performance indicators for each algorithm: the number of iterations to convergence ($m$),\footnote{It should be noted that for \KIOPS, we only track the number of iterations in a sub step if the sub step is accepted, in order to not artificially inflate the iteration count.} the CPU time in seconds (\texttt{time[s]}), the (sketched) residual norm ($\eta_m$) and \emph{relative} error norm ($\xi_m^{\text{rel}}$) at termination, the number of matrix-vector products with $A$ (\mvecs), the number of inner products with vectors of size $n$ (\nprods) and size $s$ (\sprods), the number of matrix function evaluations (\mfuns), the size of the largest matrix function that needs to be evaluated (\mfunsize) and the number of other operations with vectors of length $n$ (\nops). \FOM~was not benchmarked in this way due to its high execution time and is only used as baseline for comparison of the convergence curves. The relative error norms are computed as
\[
\xi_m^{\text{rel}} := \frac{||\vy_{\text{ref}}(T) - \vy_m(T)||}{||\vy_{\text{ref}}(T)||},
\]
where the reference solution $\vy_{\text{ref}}(T)$ was pre-computed using \FOM, requiring a residual norm $\eta_m \leq 10^{-15}$.

In the implementation of \restartrand~and \sFOM, we use the MATLAB function \texttt{expmv}, which implements the algorithm proposed in~\cite{al2011computing} and allows to efficiently compute the action of the matrix exponential on a vector over equally spaced time intervals by reusing some computations.  

All algorithms are run until $\eta_m \leq \tol = 10^{-8}$. For \KIOPS~(which is not a residual based method), this is enforced by simply checking if $\eta_m \leq \tol$ after termination and supplying a smaller tolerance for the error estimate if $\eta_m > \tol$.

All experiments are carried out in MATLAB 25.1 (R2025a) on a PC with an Intel Core Ultra 7 165U CPU with a clock rate of 4.9 GHz and 32 GB RAM. The MATLAB code to reproduce the experiments can be found at \url{https://github.com/kriegerbuw/ks-res-ode}.

\begin{figure}
\centering
\input{fig/convdiff_time_t1.tikz}
\caption{Execution time for 3d convection diffusion problem with $T = 1$.}
\label{fig:convdiff time}
\end{figure}
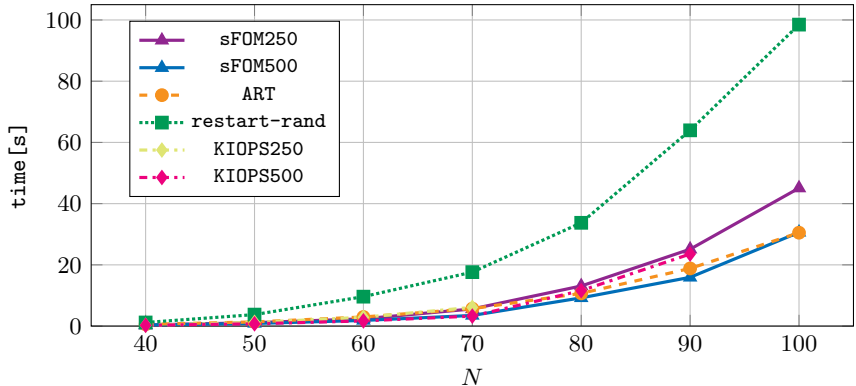

\subsection{Three-dimensional convection-diffusion equation}\label{subsec:conv-diff}
As a first model problem, we consider the semi-discretization of the three-dimensional convection-diffusion equation
\begin{equation*}\label{eq:condiff_op}
    u_t = -\nu\Delta u+w\nabla u,
\end{equation*}
on $[0, T] \times [0,1]^3$ by centered finite differences with $N$ grid points in each spatial dimension. The convection field is chosen as $w=(x\sin(x),y\cos(y),e^{z^2-1})$ and the diffusion coefficient is $\nu=5\cdot 10^{-3}$. We add a constant inhomogeneity $\vg$, which is a vector containing the values of the function $10e^{-100((x-\frac{1}{2})^2+(y-\frac{1}{2})^2+(z-\frac{1}{2})^2)}$ evaluated at the grid points, and choose the initial value $\vb_0$ as a vector with normally distributed entries, normalized to have a norm of 1. This leads to the ODE IVP
\[
\begin{cases}
    \begin{aligned}
        \vy^\prime(t) &= -A\vy(t) + \vg \text{ on }(0,T], \\
        \vy(0) &= \vb_0,
    \end{aligned} 
\end{cases}
\]
with solution
\[
\vy(t) = \exp(-tA)\vb + t\varphi_1(-tA)\vg = t\varphi_1(-tA)(\vg-A\vb) + \vb.
\]
In our tests, we vary $N$ between $50$ and $150$, leading to matrices of size between $n = 125,\!000$ and $n = 3,\!375,\!000$.

\Cref{fig:convdiff time} shows the CPU times required by the different algorithms to reach the desired accuracy at $T=1$ for all considered problem instances. Across all values of $N$, \sFOMb~is competitive with \KIOPSb, the otherwise fastest tested algorithm. At $N=120$, \sFOMs~starts to become a bit slower than \KIOPSb~and \sFOMb, as it has to restart. However, we can see that it is still competitive with \KIOPSs. For $N=110$, \KIOPSs~has to employ time stepping for the first time while \sFOM~still converges in $250$ iterations or less, leading to a speedup of about $1.67$. It should, however, be noted that after time stepping, \KIOPSs~always does the full $250$ iterations and \sFOMs~therefore becomes only marginally faster than \KIOPSs~as its number of iterations approaches $500$ for the larger problems. All versions of \KIOPS~and \sFOM~are significantly faster than \phiRT~and \restartrand~for this problem.

\begin{table}[t]
\caption{Performance indicators for 3d convection diffusion problem with $N = 150$ and $T = 1$.}
\label{tab:convdiff}
\centering
\scalebox{.7}{
\begin{tabular}{c|cccccc}
& \phiRT & \restartrand & \KIOPSs & \KIOPSb & \sFOMs & \sFOMb \\
 \hline
$m$ &  $644$ & $420$ & $500$ & $356$ & $400$ & $350$ \\
\texttt{time[s]} & $92.1598$ & $153.9566$ & $43.5597$ & $32.5509$ & $34.5991$ & $28.0607$ \\
$\eta_m$ & $5.3856 \cdot 10^{-9}$ & $9.2843 \cdot 10^{-13}$ & $2.7525 \cdot 10^{-17}$ & $6.4268 \cdot 10^{-9}$ & $2.4726 \cdot 10^{-9}$ & $7.7469 \cdot 10^{-10}$ \\
$\xi_m^{\text{rel}}$ & $6.4729 \cdot 10^{-12}$ & $4.8547 \cdot 10^{-14}$ & $1.1516 \cdot 10^{-14}$ & $6.666 \cdot 10^{-13}$ & $2.2995 \cdot 10^{-13}$ & $7.3458 \cdot 10^{-14}$ \\
\mvecs & $644$ & $420$ & $500$ & $356$ & $402$ & $351$ \\
\sprods & $0$ & $12811$ & $0$ & $0$ & $85404$ & $122852$ \\
\nprods & $10536$ & $0$ & $1502$ & $1069$ & $800$ & $700$ \\
\mfuns & $1330$ & $14$ & $30$ & $26$ & $17$ & $14$ \\
\mfunsize & $31$ & $420$ & $251$ & $357$ & $250$ & $350$ \\
\nops & $10536$ & $13170$ & $1498$ & $1067$ & $1600$ & $1400$ \\
\end{tabular}
}
\end{table}

To explain the observed CPU times, we take an exemplary look at the performance indicators measured for $N=150$, which can be found in~\Cref{tab:convdiff} (results look very similar across all $N$, so we only report them for the largest value). We observe that the performance indicators for \KIOPSb~and \sFOMb~look very similar, which leads to similar CPU times, the only real difference being the large number of (very cheap) inner products with vectors of size $s$, which is somewhat counteracted by the use of fewer (expensive) inner products with vectors of size $n$. \KIOPSs~and \sFOMs~take slightly longer as they need more iterations to converge and have to restart once. Although \phiRT~requires more matrix function evaluations as well as more inner products and other vector operations of size $n$ than \KIOPS~and \sFOM, its cheap matrix function evaluations (due to extremely small compressed matrices of size at most $m = 30$) may make it competitive with \KIOPS~and \sFOM~once they have to restart a couple of times, as their restarts are much more costly. Lastly, \restartrand~has the highest CPU time across all grid sizes due to the high number of operations with vectors of length $n$ that arise in the explicit orthogonalization of the basis with respect to $\langle\cdot,\cdot\rangle_S$ and the high cost associated with the matrix function evaluations in later iterations.

\begin{figure}
\centering
\input{fig/convergence_convdiff.tikz}
\caption{Convergence curves for 3d convection diffusion problem considered in \Cref{subsec:conv-diff} with $N = 150$ and $T = 1$.}
\label{fig:convergence convdiff}
\end{figure}
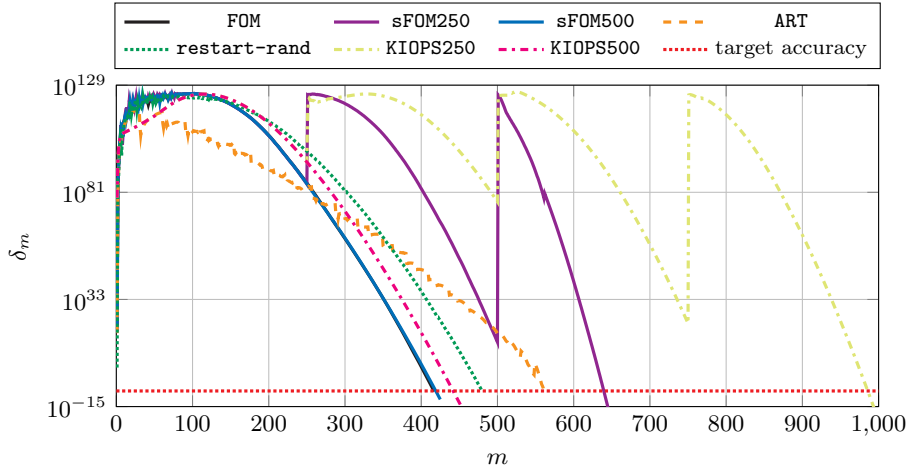

\Cref{fig:convergence convdiff} shows the convergence curves for the different algorithms for $N=150$. We observe that the sketched residual in \sFOMb~behaves almost identically to the residual of \FOM, while \KIOPSb~shows a slightly delayed convergence, as it is essentially the truncated Arnoldi algorithm at this point. Similar behavior can be observed for \sFOMs~and \KIOPSs, where convergence after restarting accelerates quickly for \sFOMs~and a bit more slowly for \KIOPSs. Due to its longer restarts, \restartrand~converges in less iterations than \phiRT, although its execution time is higher.

\subsection{Photonic crystal -- three-dimensional Maxwell equations}\label{subsec:photonic_crystal}
Our next model problem is taken from~\cite[Section~3.2]{bokn20} and was originally considered in~\cite{kole2001unconditionally}. We consider a semi-discretization of the three-dimensional Maxwell equations
\[
H_t = -\frac{1}{\mu_0} \nabla E, \qquad
E_t = -\frac{1}{\epsilon_0} \nabla H
\]
in a lossless, source-free medium. Here, $H$ denotes the magnetic field and $E$ denotes the electric field, and the permeability coefficient $\mu_0$ and permittivity coefficient $\epsilon_0$ depend on the spatial position. The setup of the domain and the boundary conditions are chosen exactly as in~\cite{bokn20}: We consider the spatial domain $\Omega = [-6.05, 6.05]^3$ which is filled with air (corresponding to relative permittivity $\epsilon_0 = 1$). Inside this domain,  there is a dielectric material (with relative permittivity $\epsilon_0 = 5$) in the region $\Omega_D = [-4.55, 4.55]^3$ which is pierced by 27 air-filled cylinders with radius $1.4$ and centers $(x_i, y_j, z_k) = (3.03i, 3.03j, 3.03k)$, where $i,j,k \in \{-1,0,1\}$. The permeability is $\mu_0 \equiv 1$ across the whole domain $\Omega$. We consider a perfectly conducting domain boundary, i.e., the boundary conditions are chosen such that the tangential electric field is zero. As initial values, we consider an all-zero magnetic field $H$ and a point light source located in the middle of the spatial domain as only non-zero values in $E$.

For the spatial discretization, we employ a finite-difference staggered Yee discretization on an $80 \times 80 \times 80$ grid (leading to a system of size $n = 3,\!188,\!646$). We seek a solution at final time $T = 10$. We obtain an ODE IVP
\[
\begin{cases}
    \begin{aligned}
        \vy^\prime(t) &= -A\vy(t) \text{ on }(0,T], \\
        \vy(0) &= \vb_0
    \end{aligned}
\end{cases}
\]
with solution
\[
\vy(t) = \exp(-tA)\vb_0.
\]

\begin{table}[t]
\caption{Performance indicators for photonic crystal problem on $80 \times 80 \times 80$ grid with $T=10$.}
\label{tab:crystal 80}
\centering
\scalebox{.7}{
\begin{tabular}{c|cccccc}
& \expRT & \restartrand & \KIOPSs & \KIOPSb & \sFOMs & \sFOMb \\
 \hline
$m$ & $735$ & $300$ & $500$ & $296$ & $325$ & $300$ \\
\texttt{time[s]} & $91.09$ & $100.3193$ & $34.2831$ & $19.5276$ & $22.1187$ & $18.0735$ \\
$\eta_m$ & $1.86 \cdot 10^{-9}$ & $8.0679 \cdot 10^{-15}$ & $2.0706 \cdot 10^{-132}$ & $1.4911 \cdot 10^{-13}$ & $2.9242 \cdot 10^{-25}$ & $5.7778 \cdot 10^{-16}$ \\
$\xi_m^{\text{rel}}$ & $8.95 \cdot 10^{-12}$ & $9.5395 \cdot 10^{-13}$ & $3.7886 \cdot 10^{-14}$ & $3.7818 \cdot 10^{-14}$ & $5.947 \cdot 10^{-12}$ & $1.7043 \cdot 10^{-13}$ \\
\mvecs & $735$ & $300$ & $500$ & $296$ & $325$ & $300$ \\
\sprods & $0$ & $9151$ & $0$ & $0$ & $68454$ & $90302$ \\
\nprods & $12015$ & $0$ & $1502$ & $889$ & $650$ & $600$ \\
\mfuns & $4434$ & $10$ & $28$ & $26$ & $14$ & $12$ \\
\mfunsize & $30$ & $300$ & $251$ & $297$ & $250$ & $300$ \\
\nops & $12015$ & $9390$ & $1498$ & $887$ & $1296$ & $1198$ \\
\end{tabular}
}
\end{table}

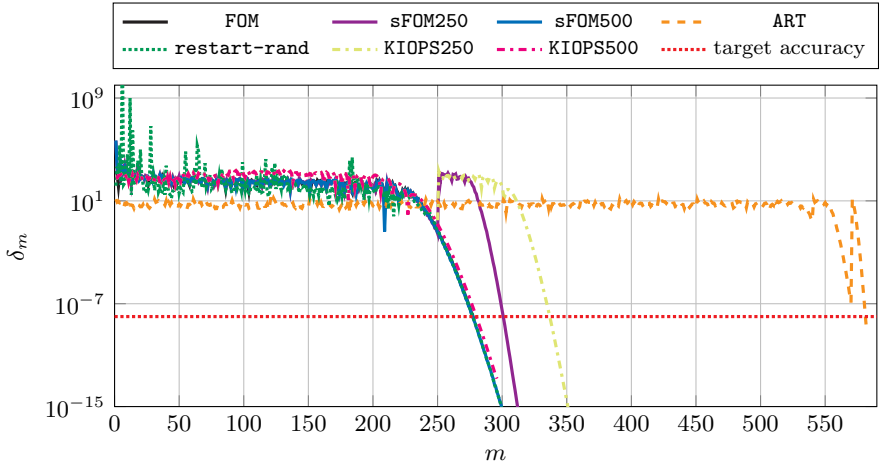
\begin{figure}
\centering
\input{fig/convergence_crystal80.tikz}
\caption{Convergence curves for photonic crystal problem on $80 \times 80 \times 80$ grid with $T=10$.}
\label{fig:convergence cyrstal 80}
\end{figure}

\Cref{tab:crystal 80} depicts the performance indicators and~\Cref{fig:convergence cyrstal 80} shows the convergence curves for this experiment. Due to the reasons already mentioned in \Cref{subsec:conv-diff}, both versions of \sFOM~are competitive with \KIOPSb, while \KIOPSs~dramatically overshoots the target accuracy as it only checks whether the target accuracy is reached at the end of every restart. \expRT~and \restartrand~are significantly slower than \KIOPS~and \sFOM, which is mostly caused by the high number of operations with vectors of size $n$, with \expRT~additionally requiring almost twice as many matrix-vector products as the other algorithms due to its delayed convergence.

\subsection{Vibrating membrane -- two-dimensional wave equation}\label{subsec:drum}
In this experiment, taken from~\cite[Section~5.2]{guidotti2025accelerating}, we simulate the vibrations of a circular membrane $\Omega$ of radius $r = 1$ that is attached to a rigid frame. At a time point $t$, the height of the membrane in a point $(x,y) \in \Omega$ is described by the solution $u$ of the wave equation
\[
u_{tt} = -\nu^2 \Delta u,
\]
where $\nu > 0$ is the speed at which transversal waves propagate in the membrane.  Using the same setup as in~\cite{guidotti2025accelerating}, we set the initial conditions to 
\[
u(x,y,0) = J_4(\eta_4\sqrt{x^2+y^2}),\qquad u_t(x,y,0) = 0
\]
where $J_4$ is the fourth-order Bessel function of the first kind and $\eta_4$ is the fourth positive root of $J_4$.

We semi-discretize the wave equation using a P1 finite element method (FEM), leading to an ODE IVP
\begin{equation}\label{eq:drum ode}
\begin{cases}
    \vy^{\prime\prime}(t) = -\nu^2 L \vy(t) \text{ on } (0,T],\\
    \vy(0) = \vb_0, \ \ \vy^\prime(0) = \vnull,
\end{cases}
\end{equation}
with $L = M^{-1}A$, where $A$ denotes the stiffness matrix and $M$ is the lumped mass matrix and the initial conditions are give by $[\vb_0]_k = J_4(\eta_4\sqrt{x_k^2+y_k^2})$, where $(x_k,y_k)$ are the coordinates of the $k$th node in the FEM mesh. When using a P1 finite element of size $2.5\cdot 10^{-3}$, this leads to a matrix of size $n = 580,\!012$. The solution to \eqref{eq:drum ode} is
\[
\vy(t) = \cos\br{\nu t\sqrt{L}}\vb_0.
\]

\begin{table}[t]
\caption{Performance indicators for the vibrating membrane problem with $T = 1.5$.}
\label{tab:drum}
\centering
\scalebox{.7}{
\begin{tabular}{c|cccccc}
& \expRT & \restartrand & \sFOMs & \sFOMb \\
 \hline
$m$ & $9765$ & $1300$ & $1850$ & $1950$ \\
\texttt{time[s]} & $504.178$ & $341.9677$ & $93.8612$ & $110.766$ \\
$\eta_m$ & $7.2574 \cdot 10^{-9}$ & $6.7072 \cdot 10^{-38}$ & $3.8647 \cdot 10^{-47}$ & $7.6118 \cdot 10^{-19}$ \\
$\xi_m^{\text{rel}}$ & $3.436 \cdot 10^{-10}$ & $4.8321 \cdot 10^{-9}$ & $3.4293 \cdot 10^{-10}$ & $3.3979 \cdot 10^{-10}$ \\
\mvecs & $9765$ & $1300$ & $3550$ & $3400$ \\
\sprods & $0$ & $65651$ & $869392$ & $1666318$ \\
\nprods & $161010$ & $0$ & $23545$ & $21277$ \\
\mfuns & $59176$ & $26$ & $137$ & $168$ \\
\mfunsize & $30$ & $1300$ & $250$ & $500$ \\
\nops & $161010$ & $66850$ & $20721$ & $19487$ \\
\end{tabular}
}
\end{table}

For \expRT~and \KIOPS, which are specifically designed to solve first-order ODEs, we transform \eqref{eq:drum ode} to a first-order problem by defining
\[
\vz(t) = \begin{bmatrix}
    \vy(t) \\ \vy^\prime(t)
\end{bmatrix}, \qquad \mathcal{L} = \begin{bmatrix}
    0 & -I \\ \nu^2L & 0
\end{bmatrix}, \qquad \vz_0 = \begin{bmatrix}
    \vb_0 \\ \vnull
\end{bmatrix}.
\]
This yields
\begin{equation}\label{eq:first_order_formulation}
\begin{cases}
\vz^\prime(t) = -\mathcal{L}\vz(t)  \text{ on } (0,T],\\
\vz(0) = \vz_0,
\end{cases}
\end{equation}
with solution
\[
\vz(t) = \exp\br{-t\mathcal{L}}\vz_0.
\]
Note that while this first-order formulation requires working with vectors of length $2n$ (which, e.g., increases the orthogonalization cost and the memory required to store the Krylov basis), matrix vector products with $\mathcal{L}$ can be performed at essentially the same cost as those with $L$, via
\[
\begin{bmatrix}
    \vv^{(1)} \\ \vv^{(2)}
\end{bmatrix} \mapsto \begin{bmatrix}
    -\vv^{(2)} \\ \nu^2L\vv^{(1)}
\end{bmatrix}.
\]

We also use the equivalent first-order formulation~\eqref{eq:first_order_formulation} for solving the projected IVP in  \restartrand~and \sFOM, as this led to improved stability in preliminary experiments.

\Cref{tab:drum} shows the performance indicators and \Cref{fig:convergence drum} shows the convergence curves from our experiments where we choose $\nu = 0.85$ and $T = 1.5$. Note that since both \KIOPSs~and \KIOPSb~failed to converge, they are not listed in the table\footnote{KIOPS was also tested with other maximum Krylov dimensions---up to $m_{\max} = 1500$---of which all either produced NaNs or reduced the time interval to below machine precision.} and \expRT~is not plotted as it behaves very similarly as in the other experiments but makes \Cref{fig:convergence drum} substantially less readable. For the second order problem, \sFOM~was up to five times as fast as \expRT~and \restartrand. It should be noted that when doing an RT-restart with time step $\tau$ we in general have $\vu_m^\prime(\tau) \neq \vnull$. Therefore, after restarting, \sFOM~employs a block Krylov approach (cf.~\Cref{example:block_krylov}). This increases \mvecs~by a factor of two as well as \sprods~and \nprods~by a factor of four after the first restart, but has a much milder impact on the computation time, as all corresponding operations are blocked. Due to the high number of iterations, \restartrand~has to evaluate functions of several comparatively large matrices, up to a matrix dimension of 1200, which becomes extremely costly in addition to the high number of \nops.

\begin{figure}
    \centering
    \input{fig/convergence_drum.tikz}
    \caption{Convergence curves for the vibrating membrane problem with $T = 1.5$.}
    \label{fig:convergence drum}
\end{figure}
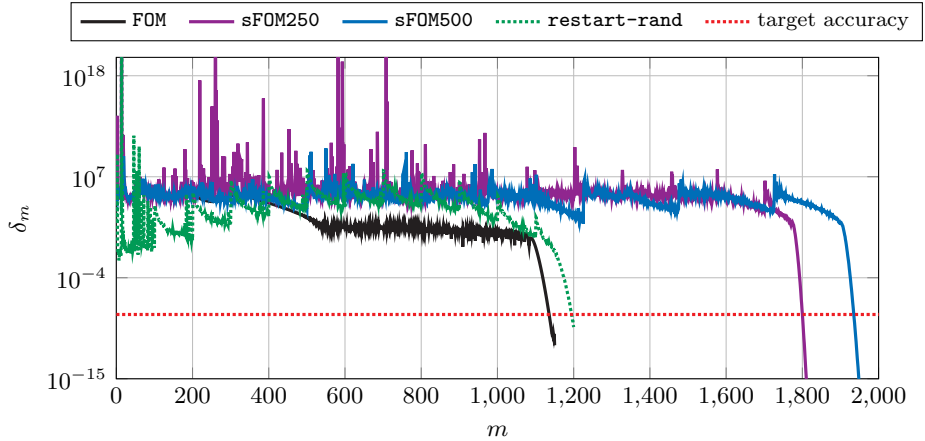

We also see that \sFOMb~takes more iterations to converge than \sFOMs, leading to very similar numbers in most performance indicators and \sFOMs~being about 20 seconds faster. The reason for this is that in their first (non-blocked) restarts, both algorithms are able to reduce the residual to just slightly above the desired tolerance of $10^{-8}$ for $30-40\%$ of the time interval but only actually reach the tolerance for $<1\%$ of the interval, leading to extremely costly first restarts for both. After the restart, however, \sFOMb~converges slightly faster than \sFOMs, as can be seen by it not needing the entire 250 iterations ``lost'' in the first restart more than \sFOMs.

\section{Conclusions \& Outlook} \label{sec:conclusions}
We have introduced a general framework that specifies conditions under which the Krylov ODE residual is easily and efficiently computable, which encompasses both polynomial and rational Krylov methods employing arbitrary inner products. As a particularly important special case that is covered by our framework, we have discussed Krylov methods with randomized sketching in depth. In particular, our framework allows us to derive a reliable, rigorous stopping criterion for the sketched Arnoldi method in the context of solving ODEs, while so far, only heuristic stopping criteria have been available in the literature. Additionally, we have incorporated residual-time restarting into the sketched Arnoldi method, to make it more stable and efficient. Numerical experiments on large-scale real-world problems demonstrate the use of the sketched residual as a stopping criterion for the sketched Arnoldi method as well as the competitiveness of sketched Arnoldi with several state-of-the-art methods.

As a future research goal, it would be valuable to derive rigorous relations between the residuals of the sketched Arnoldi method and the standard (full) Arnoldi method. This would allow us to more directly characterize possible delays of convergence introduced by sketching. Another interesting area for future research is the extension of our methodology to \emph{fractional} differential equations, whose solutions can be expressed in terms of Mittag--Leffler functions~\cite{haubold2011mittag}.

\paragraph{Acknowledgment}
We are indebted to Mike Botchev and Leonid Knizhnerman for providing us the matrix and initial conditions used in the experiment in~\Cref{subsec:photonic_crystal}.

\bibliography{matrixfunctions}
\bibliographystyle{siam}

\end{document}

%% file: setup.tex
\usepackage{lineno,hyperref}
\modulolinenumbers[5]
\usepackage{amsmath}
\usepackage{amsfonts,amssymb}
\usepackage{graphicx}
\usepackage{bm}
\graphicspath{{./Figures/}}
\usepackage{color}
\usepackage{ulem}
\usepackage{multirow}
\usepackage{multicol}
\usepackage{algorithm}
\usepackage{algpseudocode}
\usepackage{algorithmicx}
\usepackage{hyperref}
\usepackage{upquote}
\usepackage{cprotect}
\usepackage{setspace}
\usepackage{subcaption}
\usepackage{cleveref}

\usepackage[showonlyrefs]{mathtools}
\usepackage{xcolor}

\usepackage[defaultcolor=red]{changes}

\DeclareMathAlphabet{\mathpzc}{OT1}{pzc}{m}{it}
\DeclareMathAlphabet{\mathbf}{OT1}{cmr}{bx}{it}
\newcommand{\va}{{\mathbf a}}
\newcommand{\vb}{{\mathbf b}}
\newcommand{\vc}{{\mathbf c}}

\newcommand{\ve}{{\mathbf e}}

\newcommand{\vf}{{\mathbf f}}

\newcommand{\vg}{{\mathbf g}}

\newcommand{\vh}{{\mathbf h}}

\newcommand{\vr}{{\mathbf r}}
\newcommand{\vs}{{\mathbf s}}

\newcommand{\vu}{{\mathbf u}}
\newcommand{\vv}{{\mathbf v}}

\newcommand{\vw}{{\mathbf w}}
\newcommand{\vx}{{\mathbf x}}
\newcommand{\vy}{{\mathbf y}}
\newcommand{\vz}{{\mathbf z}}

\newcommand{\vnull}{\boldsymbol{0}}

\newcommand{\spK}{{\cal K}}

\newcommand{\spQ}{{\cal Q}}

\renewcommand{\d}{\,\mathrm{d}}
\newcommand{\N}{\mathbb{N}}
\newcommand{\R}{\mathbb{R}}
\newcommand{\Rn}{\mathbb{R}^n}

\newcommand{\Rnn}{\mathbb{R}^{n \times n}}

\newcommand{\C}{\mathbb{C}}

\newcommand{\V}{\mathcal{V}}

\DeclareMathOperator{\Span}{span}
\DeclareMathOperator{\colspan}{colspan}

\newcommand{\Rsq}[1]{\mathbb{R}^{#1 \times #1}}
\newcommand{\Rg}[2]{\mathbb{R}^{#1 \times #2}}

\newcommand{\br}[1]{\mathopen{}\left( #1 \right)\mathclose{}}
\newcommand{\bbr}[1]{\mathopen{}\left[ #1 \right]\mathclose{}}

\newcommand{\cbr}[1]{\mathopen{}\left\{ #1 \right\}\mathclose{}}
\newcommand{\abs}[1]{\mathopen{}\left\vert #1 \right\vert\mathclose{}}
\newcommand{\norm}[2]{\mathopen{}\left\Vert #1 \right\Vert_{#2}\mathclose{}}

\newcommand{\OO}[1]{\mathcal{O}\br{#1}}

\newcommand{\thetamax}{\max_{0\leq\theta\leq t}}

\newcommand{\errnormt}{\norm{\vxi_m\br{t}}{}}
\newcommand{\errhatnormt}{\|\vxi_m\br{t}\|}

\newcommand{\inv}{^{-1}}
\newcommand{\tp}{^{\top}}

\newcommand{\her}{^{\ast}}

\newcommand{\kry}[2]{\mathcal{K}_{#1}\br{#2}}
\newcommand{\kryab}[1]{\kry{#1}{A,\vb}}
\newcommand{\sprod}[2]{\langle #1, #2 \rangle}

\newcommand{\uH}{\underline{H}}
\newcommand{\uG}{\underline{G}}
\newcommand{\uK}{\underline{K}}
\newcommand{\vxi}{{\boldsymbol{\xi}}}
\newcommand{\Hhat}{\widehat{H}}

\newcommand{\Chat}{\mathpzc{C}}

\newcommand{\ym}[1]{\vy_m\br{#1}}
\newcommand{\ymt}{\vy_m\br{t}}

\newcommand{\xmt}{\vx_m\br{t}}
\newcommand{\rmt}{\vr_m\br{t}}

\newcommand{\xhatmt}{\vx_m\br{t}}
\newcommand{\rhatmt}{\vr_m\br{t}}

\newcommand{\false}{\mathtt{false}}
\newcommand{\true}{\mathtt{true}}
\newcommand{\update}{\mathtt{update}}
\newcommand{\slope}{\sigma}
\newcommand{\tol}{\mathtt{tol}}
\newcommand{\rtol}{\mathtt{rtol}}
\newcommand{\done}{\mathtt{converged}}

\newcommand{\bbeta}{\boldsymbol{\beta}}

\newcommand{\wQ}{\mathpzc{Q}}

\newcommand{\betahatmt}{\bbeta_m\br{t}}

\newcommand{\pV}{\mathpzc{V}}
\newcommand{\pU}{\mathpzc{U}}

\newcommand{\pB}{\mathpzc{B}}

\newcommand{\pp}{\mathpzc{p}}
\newcommand{\ph}{\mathpzc{h}}
\newcommand{\pE}{\mathpzc{E}}

\newcommand{\pg}{\mathpzc{g}}

\newcommand{\pR}{\mathpzc{R}}

\newcommand{\FOM}{\texttt{FOM}}
\newcommand{\phiRT}{\texttt{phiRT30}}
\newcommand{\expRT}{\texttt{expRT30}}
\newcommand{\restartrand}{\texttt{restart-rand}}
\newcommand{\KIOPS}{\texttt{KIOPS}}
\newcommand{\sFOM}{\texttt{sFOM}}
\newcommand{\KIOPSs}{\texttt{KIOPS250}}
\newcommand{\sFOMs}{\texttt{sFOM250}}
\newcommand{\KIOPSb}{\texttt{KIOPS500}}
\newcommand{\sFOMb}{\texttt{sFOM500}}
\newcommand{\mvecs}{\texttt{mvecs}}
\newcommand{\sprods}{\texttt{s-prods}}
\newcommand{\nprods}{\texttt{n-prods}}
\newcommand{\mfuns}{\texttt{mfuns}}
\newcommand{\mfunsize}{\texttt{mfunsize}}
\newcommand{\nops}{\texttt{n-ops}}

\newtheorem{remarksimple}[theorem]{Remark}
\let\oldremarksimple\remarksimple
\renewcommand{\remarksimple}{\oldremarksimple\normalfont}

\newtheorem{experimentsimple}[theorem]{Experiment}
\let\oldexperimentsimple\experimentsimple
\renewcommand{\experimentsimple}{\oldexperimentsimple\normalfont}

\newtheorem{examplesimple}[theorem]{Example}
\let\oldexamplesimple\examplesimple
\renewcommand{\examplesimple}{\oldexamplesimple\normalfont}
\newenvironment{example}{\begin{examplesimple}}{\hfill$\diamond$\end{examplesimple}}

\usepackage{pgfplots}
\usepgfplotslibrary{groupplots}
\usetikzlibrary{external}

\pgfkeys{/pgf/images/include external/.code={\includegraphics{#1}}}

\usetikzlibrary{decorations.markings}
\makeatletter
\tikzset{
  nomorepostactions/.code={\let\tikz@postactions=\pgfutil@empty},
  mymark/.style 2 args={decoration={markings,
    mark= between positions 0 and 1 step (1/9)*\pgfdecoratedpathlength with{%
        \tikzset{#2,every mark}\tikz@options
        \pgfuseplotmark{#1}%
      },  
    },
    postaction={decorate},
    /pgfplots/legend image post style={
        mark=#1,mark options={#2},every path/.append style={nomorepostactions}
    },
  },
}
\makeatother

\pgfplotscreateplotcyclelist{list_bounds}{%
Cerulean,line width=1.25pt, mark=*, solid, mark repeat=10\\
BurntOrange,line width=1.25pt, mark=diamond, dashed, mark repeat=10, mark options={solid}\\
OliveGreen,line width=1.25pt, mark=x, dashdotted, mark repeat=10, mark options={solid}\\
Orchid,line width=1.25pt, mark=square, dashdotted, mark repeat=10, mark options={solid}\\
BrickRed,line width=1.25pt, mark=o, dashdotted, mark repeat=10, mark options={solid}\\
LimeGreen,line width=1.25pt, mark=triangle, dashdotted, mark repeat=10, mark options={solid}\\
}
\pgfplotsset{compat=1.17}

\Crefname{remarksimple}{Remark}{Remarks}
\Crefname{subsection}{Section}{Sections}

\algnewcommand{\Input}{\Statex \textbf{Input:} }
\algnewcommand{\Output}{\Statex \textbf{Output:} }

%% file: fig/err_res_bound_grouped.tikz.tex
\pgfplotsset{height=0.35\linewidth,
            width=0.5\linewidth,
            compat=1.17,
            every axis/.append style={
                        legend style={/tikz/every even column/.append style={column sep=6pt}}, font=\footnotesize}}

\noindent%
\begin{tikzpicture}[scale=1]%
\begin{groupplot}[
    group style={group size=2 by 2,horizontal sep=10pt,vertical sep=10pt},
    legend style={at={(.0 5,.05)}, anchor=south west},
    legend columns=1,
    xmin=0,
    xmax=263,
    grid=major,
    xlabel={\small $m$},
    ymin = 1e-16,
    ymax = 1e4,
    tick label style={font=\small}]
\nextgroupplot[ymode=log,xlabel=\empty]
    \addplot[color=Black,very thick,dashed]
    	    table [x ={m},y ={err_fom}] {data/compare_err_res.dat};
            \addlegendentry{FOM error}
    \addplot[color=Red,very thick]
	    table [x ={m},y ={bound_fom}] {data/compare_err_res.dat};
        \addlegendentry{Bound~\eqref{eq:bound_error_phi2}}
        
\nextgroupplot[ymode=log,yticklabel=\empty,xlabel=\empty]
    \addplot[color=Black,very thick,dashed]
	    table [x ={m},y ={err_sfom}] {data/compare_err_res.dat};
        \addlegendentry{sFOM error}

    \addplot[color=Red,very thick]
	    table [x ={m},y ={bound_sfom}] {data/compare_err_res.dat};
        \addlegendentry{Bound~\eqref{eq:error_bound_sketched}}
        
\end{groupplot}
\end{tikzpicture}

%% file: fig/convdiff_time_t1.tikz.tex
\pgfplotsset{height=0.45\linewidth,
            width=0.9\linewidth,
            compat=1.17,
            every axis/.append style={
                        legend style={/tikz/every even column/.append style={column sep=6pt}}, font=\footnotesize}}

\noindent%
\begin{tikzpicture}[scale=1]%
    \begin{semilogyaxis}[
        legend style={at={(.95,.05)}, anchor=south east},
        legend columns=2,
        xmin=45,
        xmax=155,
        grid=major,
        xlabel={\small $N$},
        ylabel={\small \texttt{time[s]}},
        ymin = 1e-1,
        ymax = 200,
        tick label style={font=\small}]
    \addplot[color=Plum,very thick,mark=triangle*]
	    table [x ={n},y ={t_sfom250}] {data/time/convdiff_t1.0_median.dat};
        \addlegendentry{\sFOMs}
        
    \addplot[color=NavyBlue,very thick,mark=triangle*]
	    table [x ={n},y ={t_sfom500}] {data/time/convdiff_t1.0_median.dat};
        \addlegendentry{\sFOMb}

    \addplot[color=BurntOrange,very thick,dashed,mark=*,mark options=solid]
        table [x ={n},y ={t_art}] {data/time/convdiff_t1.0_median.dat};
        \addlegendentry{\phiRT}

    \addplot[color=ForestGreen,very thick,densely dotted,mark=square*,mark options=solid]
        table [x ={n},y ={t_restart_rand}] {data/time/convdiff_t1.0_median.dat};
        \addlegendentry{\restartrand}
    
    \addplot[color=GreenYellow,very thick,dashdotted,mark=diamond*,mark options=solid]
        table [x ={n},y ={t_kiops250}] {data/time/convdiff_t1.0_median.dat};
        \addlegendentry{\KIOPSs}
    
    \addplot[color=RubineRed,very thick,dashdotted,mark=diamond*,mark options=solid]
        table [x ={n},y ={t_kiops500}] {data/time/convdiff_t1.0_median.dat};
        \addlegendentry{\KIOPSb}
    \end{semilogyaxis}
       
\end{tikzpicture}

%% file: fig/convergence_convdiff.tikz.tex
\pgfplotsset{height=0.45\linewidth,
            width=0.9\linewidth,
            compat=1.17,
            every axis/.append style={
                        legend style={/tikz/every even column/.append style={column sep=6pt}}, font=\footnotesize}}

\noindent%
\begin{tikzpicture}[scale=1]%
    \begin{semilogyaxis}[
        legend style={at={(-.0025,1.05)}, anchor= south west},
        legend columns=4,
        xmin=0,
        xmax=660,
        grid=major,
        xlabel={\small $m$},
        ylabel={\small $\eta_m$},
        ymin = 1e-12,
        ymax = 1e7,
        tick label style={font=\small}]

    \addplot[color=Black,very thick]
	    table [x ={m},y ={res_fom}] {data/convergence/convdiff_t1.0_convergence.dat};
        \addlegendentry{\texttt{FOM}}

    \addplot[color=Plum,very thick]
	    table [x ={m},y ={res_sfom250}] {data/convergence/convdiff_t1.0_convergence.dat};
        \addlegendentry{\sFOMs}

        \addplot[color=NavyBlue,very thick]
	    table [x ={m},y ={res_sfom500}] {data/convergence/convdiff_t1.0_convergence.dat};
        \addlegendentry{\sFOMb}

    \addplot[color=BurntOrange,very thick,dashed]
        table [x ={m},y ={res_art}] {data/convergence/convdiff_t1.0_convergence.dat};
        \addlegendentry{\phiRT}

    \addplot[color=ForestGreen,very thick,densely dotted]
        table [x ={m},y ={res_restart_rand}] {data/convergence/convdiff_t1.0_convergence.dat};
        \addlegendentry{\restartrand}

    \addplot[color=GreenYellow,very thick,dashdotted]
        table [x ={m},y ={res_kiops250}] {data/convergence/convdiff_t1.0_convergence.dat};
        \addlegendentry{\KIOPSs}
    
    \addplot[color=RubineRed,very thick,dashdotted]
        table [x ={m},y ={res_kiops500}] {data/convergence/convdiff_t1.0_convergence.dat};
        \addlegendentry{\KIOPSb}

    \addplot[color=Red,very thick,densely dotted]
        coordinates {(0,1e-8)
    (1000,1e-8)};
        \addlegendentry{target accuracy}

    \end{semilogyaxis}
\end{tikzpicture}

%% file: fig/convergence_crystal80.tikz.tex
\pgfplotsset{height=0.45\linewidth,
            width=0.9\linewidth,
            compat=1.17,
            every axis/.append style={
                        legend style={/tikz/every even column/.append style={column sep=6pt}}, font=\footnotesize}}

\noindent%
\begin{tikzpicture}[scale=1]%
    \begin{semilogyaxis}[
        legend style={at={(-.0025,1.05)}, anchor= south west},
        legend columns=4,
        xmin=0,
        xmax=750,
        grid=major,
        xlabel={\small $m$},
        ylabel={\small $\eta_m$},
        ymin = 1e-15,
        ymax = 1e10,
        tick label style={font=\small}]

    \addplot[color=Black,very thick]
	    table [x ={m},y ={res_fom}] {data/convergence/crystal_80_t10.0_convergence.dat};
        \addlegendentry{\texttt{FOM}}

    \addplot[color=Plum,very thick]
	    table [x ={m},y ={res_sfom250}] {data/convergence/crystal_80_t10.0_convergence.dat};
        \addlegendentry{\sFOMs}

        \addplot[color=NavyBlue,very thick]
	    table [x ={m},y ={res_sfom500}] {data/convergence/crystal_80_t10.0_convergence.dat};
        \addlegendentry{\sFOMb}

    \addplot[color=BurntOrange,very thick,densely dashed]
        table [x ={m},y ={res_art}] {data/convergence/crystal_80_t10.0_convergence.dat};
        \addlegendentry{\expRT}

    \addplot[color=ForestGreen,very thick,densely dotted]
        table [x ={m},y ={res_restart_rand}] {data/convergence/crystal_80_t10.0_convergence.dat};
        \addlegendentry{\restartrand}

    \addplot[color=GreenYellow,very thick,dashdotted]
        table [x ={m},y ={res_kiops250}] {data/convergence/crystal_80_t10.0_convergence.dat};
        \addlegendentry{\KIOPSs}
    
    \addplot[color=RubineRed,very thick,dashdotted]
        table [x ={m},y ={res_kiops500}] {data/convergence/crystal_80_t10.0_convergence.dat};
        \addlegendentry{\KIOPSb}

    \addplot[color=Red,very thick,densely dotted]
        coordinates {(0,1e-8)
    (1000,1e-8)};
        \addlegendentry{target accuracy}

    \end{semilogyaxis}
\end{tikzpicture}

%% file: fig/convergence_drum.tikz.tex
\pgfplotsset{height=0.45\linewidth,
            width=0.9\linewidth,
            compat=1.17,
            every axis/.append style={
                        legend style={/tikz/every even column/.append style={column sep=6pt}}, font=\footnotesize}}

\noindent%
\begin{tikzpicture}[scale=1]%
    \begin{semilogyaxis}[
        legend style={at={(-.06,1.05)}, anchor= south west},
        legend columns=5,
        xmin=0,
        xmax=2000,
        grid=major,
        xlabel={\small $m$},
        ylabel={\small $\eta_m$},
        ymin = 1e-15,
        ymax = 1e20,
        tick label style={font=\small}]


    \addplot[color=Black,very thick]
	    table [x ={m},y ={res_fom}] {data/convergence/drum_t1.275_convergence.dat};
        \addlegendentry{\texttt{FOM}}

    \addplot[color=Plum,very thick]
	    table [x ={m},y ={res_sfom250}] {data/convergence/drum_t1.275_convergence.dat};
        \addlegendentry{\sFOMs}

        \addplot[color=NavyBlue,very thick]
	    table [x ={m},y ={res_sfom500}] {data/convergence/drum_t1.275_convergence.dat};
        \addlegendentry{\sFOMb}

    \addplot[color=ForestGreen,very thick,densely dotted]
        table [x ={m},y ={res_restart_rand}] {data/convergence/drum_t1.275_convergence.dat};
        \addlegendentry{\restartrand}

    \addplot[color=Red,very thick,densely dotted]
        coordinates {(0,1e-8)
    (6000,1e-8)};
        \addlegendentry{target accuracy}

    \end{semilogyaxis}
\end{tikzpicture}